\documentclass[pdflatex,sn-mathphys-num]{sn-jnl}

\usepackage{graphicx}%
\usepackage{multirow}%
\usepackage{amsmath,amssymb,amsfonts}%
\usepackage{amsthm}%

\usepackage[title]{appendix}%
\usepackage{xcolor}%
\usepackage{textcomp}%
\usepackage{manyfoot}%
\usepackage{booktabs}%
\usepackage{algorithm}%
\usepackage{algorithmicx}%
\usepackage{algpseudocode}%
\usepackage{listings}
\theoremstyle{thmstyleone}%
\newtheorem{theorem}{Theorem}
\theoremstyle{thmstyletwo}%
\newtheorem{example}{Example}%
\newtheorem{remark}{Remark}%

\theoremstyle{thmstylethree}%
\newtheorem{definition}{Definition}%

\begin{document}

\title[Article Title]{Transcendental meromorphic solutions and the complex Schr\"{o}dinger equation with delay}


\newtheorem{corollary}[theorem]{Corollary}
\newtheorem{lemma}[theorem]{Lemma}
\newtheorem{question}[theorem]{Question}
\newtheorem{acknowledgements}[theorem]{Acknowledgements}

\author[1]{\fnm{Tingbin} \sur{Cao}}\email{tbcao@ncu.edu.cn}

\author[2]{\fnm{Risto} \sur{Korhonen}}\email{risto.korhonen@uef.fi} 

\author*[2]{\fnm{Wenlong} \sur{Liu}}\email{wenlong.liu@uef.fi}

\affil[1]{\orgdiv{Department of Mathematics}, \orgname{Nanchang University}, \orgaddress{\street{}\city{Nanchang} \postcode{330031}, \country{P. R.  China}}}

\affil*[2]{\orgdiv{Department of Physics and Mathematics}, \orgname{University of Eastern Finland}, \orgaddress{\street{} \city{P.O.Box 111}, \postcode{FI-80101 Joensuu}, \state{} \country{Finland}}}


\abstract{In this article, we focus on studying the differential-difference equation
\begin{equation}\label{E1}
f'(z) = a(z)f(z+1) + R(z, f(z)), \quad R(z, f(z)) = \frac{P(z, f(z))}{Q(z, f(z))},    
\end{equation}
where the two nonzero polynomials \( P(z, f(z)) \) and \( Q(z, f(z)) \) in \( f(z) \), with small meromorphic coefficients, are coprime, and \( a(z) \) is a nonzero small meromorphic function of \( f(z) \).  This equation includes the complex Schrödinger equation with delay as a special case.

If \( f(z) \) is a transcendental meromorphic solution of \eqref{E1} with subnormal growth, then we derive all possible forms of \eqref{E1}. Additionally, under the above assumptions, we classify these specific forms based on the degrees of \( P(z, f(z)) \) and \( Q(z, f(z)) \) to establish necessary conditions for the existence of transcendental meromorphic solutions of \eqref{E1}. In particular, when \( \deg_{f}(P) - \deg_{f}(Q) = 2 \), we demonstrate that \eqref{E1} reduces to a Riccati differential equation. Finally, examples are provided to support our results.}

\keywords{Delay differential equation, Exponential polynomial, Riccati differential equation, Small function}



\maketitle

\section{Introduction}\label{sec1}
Time delay is a common phenomenon occurring in many engineering applications. In control theory, the process of sampled-data control is a typical example, where delay occurs during transmission from measurement to controller.

As is well-known, the Schrödinger differential equation is a fundamental equation in physics, and many researchers have conducted extensive studies on it, such as investigating solution methods for Schrödinger equation-related problems (e.g., \citep{bib2,bib3,bib4,bib5,bib12,bib13,bib20,bib35}) or studying the initial value problem for the Schrödinger equation in certain linear spaces.

Agirseven \cite{bib1} introduces the following Schrödinger differential-difference equation:
\begin{align}
\begin{cases}
i\frac{dv(t)}{dt} + Av(t) = bAv(t - w) + f(t), \quad 0 < t < \infty,\\
v(t) = \varphi(t), \quad -w \le t \le 0,
\end{cases}
\end{align}
where \( A \) is a self-adjoint positive definite operator, and \( \varphi(t) \) and \( f(t) \) are continuous functions. The author establishes a stability estimation theorem for solutions to the Schrödinger differential-difference equation by studying the initial value problem for the delay Schrödinger differential equation in Hilbert space \cite{bib1}.

According to the above Schrödinger differential-difference equation, the following delay differential equation \eqref{E2} can be referred to as the complex Schr\"{o}dinger equation with delay:
\begin{equation}\label{E2}
{f}'(z) = a(z)f(z + n) + b(z)f(z) + c(z),
\end{equation}
where \( n \in \mathbb{C} \setminus \{ 0 \}. \)

Gross and Yang \cite{bib14} have investigated the rate of growth  of functions meromorphic in the plane (\( |z| < \infty \)) that are solutions of differential equations of the form
\begin{align}\label{E111} 
p(z, f(z), f'{(z)}, \dots, f^{(k)}(z)) = f(g(z)),
\end{align}
where \( p(z) \) is a polynomial in the variables \( z, f(z), f'(z), \dots, f^{(k)}(z) \), and \( g(z) \) is a given transcendental entire function.

 Brunt, Marshall and Wake \cite{bib6} considered a generalization of the pantograph equation:
\begin{align}\label{E109}
{f}'(z) = \lambda f(g(z)) - b f(z),
\end{align}
where \( b \) and \( \lambda \) are complex constants with \( \lambda \neq 0 \), and \( g \) is an entire function with a fixed point at \( z = z_0 \). Clearly, if \( g(z) = z + n \), \( n \in \mathbb{C} \setminus \{0\} \), then \eqref{E109} is a special case of \eqref{E2} with \( c(z) \equiv 0 \).

In \cite{bib30}, Li and Saleeby investigated solutions of functional-differential equations of the form
\begin{align}\label{E6}
{f}'(z) = a(z) f(g(z)) + b(z) f(z) + c(z)
\end{align}
in both real and complex variables. They found that if \( f(z) \) is a transcendental meromorphic solution of \eqref{E6} and \( g(z) \) is an entire function, then \( g(z) = A z + B \), where \( A \) is a nonzero constant and \( B \) is a constant. Clearly, if \( A = 1 \) and \( B = n \), then \eqref{E6} becomes \eqref{E2}.

This leads us to the study of transcendental meromorphic solutions of \eqref{E2} with small meromorphic coefficients, and the study of its generalized form from the perspective of Nevanlinna theory. In particular, the main purpose of this paper is to study transcendental meromorphic solutions of the following differential-difference equation:
\begin{equation} \label{E3}
{f}'(z) = a(z) f(z + n) + R(z, f(z)), \quad R(z, f(z)) = \frac{P(z, f(z))}{Q(z, f(z))},
\end{equation}
where \( a(z) (\not \equiv 0) \) is a rational function and all coefficient functions of \( R(z, f(z)) \) are small meromorphic functions with respect to \( f(z) \), and \( P(z, f(z)) \) and \( Q(z, f(z)) \) are polynomials in \( f(z) \); \( n \) is a nonzero complex number. 

When \( R(z, f(z)) \) degenerates into a polynomial of \( \deg_{f}(R) = 1 \), \eqref{E3} reduces to the complex Schr\"{o}dinger equation with delay \eqref{E2}. We observe that if \( a(z) \equiv 0 \), then \eqref{E3} can be written as
\begin{equation}
{f}'(z) = R(z, f(z)) \label{E4}.
\end{equation}
Equation \eqref{E4} includes the Riccati equation as a special case. The celebrated Malmquist theorem states that a differential equation of the form \eqref{E4}, where the right-hand side is rational in both arguments and admits a transcendental meromorphic  solution, can be reduced to a Riccati differential equation
\[
{f}'(z) = a_{0}(z) + a_{1}(z) f(z) + a_{2}(z) f(z)^2
\]
with rational coefficients. For more detailed research on equation \eqref{E4}, we refer to \cite{bib27} and \cite{bib40}. 

Throughout this paper, we assume \( a(z) \not\equiv 0 \). Without loss of generality, we set \( n = 1 \) in \eqref{E3}.

A range of results has been obtained by many researchers in the study of meromorphic  solutions of delay differential equations in the sense of Nevanlinna theory. For example, in 2017, Halburd and the second author \cite{bib16} developed an iterative method for poles and zeros to study the growth of meromorphic solutions of the following delay differential equation.
\begin{theorem}
\cite[Theorem 1.1]{bib16} $$w(z+1)-w(z-1)+a(z)\frac{{w}'(z)}{w(z)}=R(z,w(z))=\frac{P(z,w(z))}{Q(z,w(z))},$$
where $a(z)$ is rational, 
$P(z,w)$ is a polynomial in $w$ having rational coefficients in $z$, and $Q(z, w)$ is a polynomial in $w(z)$ with roots that are nonzero rational functions of $z$ and not roots of $P(z,w)$. If the hyper-order of $w(z)$ is less than one, then $$\deg_{w}(P)=\deg_{w}(Q)+1\le 3\quad   \text{or}\quad \deg_{w}(R)\le 1.$$
\end{theorem}

In 2022, Lü \cite{bib34} generalized \cite[Theorem 2.4]{bib30} and studied meromorphic  solution of \eqref{E9} by using Nevanlinna theory.

\begin{theorem}
\cite[Theorem 2]{bib34}
Suppose $f(z)$ is a transcendental meromorphic function in $\mathbb{C} $ and $g$ is a non-constant entire function satisfying the equation 
\begin{align}\label{E9}
f^{(k)}(z)=a(z)(f^{n}\circ g)(z)+b(z)f(z)+c(z),
\end{align}
where $a(\not=0),$ $b,$ $c$ are polynomials in $\mathbb{C}.$
Then $g$ must be linear. Further, \eqref{E9} does not admit any transcendental meromorphic solution if $n\ge2$  
and $\frac{k}{n-1}$ is not an integer.

\end{theorem}
In recent years, Xu and the fist author \cite{bib39}, the first author, Chen and the second author \cite{bib7}, Laine and Latreuch \cite{bib28}, the first author  and Chen \cite{bib8},
Zhang and Huang \cite{bib41}, Liu  and Song  \cite{bib31}, Hu and Liu \cite{bib24}, Wang, Han and Hu \cite{bib37} have also used the   Nevanlinna theory of meromorphic functions to study the meromorphic solutions of delay differential equations in the complex plane $\mathbb{C}$.

\section{Main Results}\label{sec2}
\begin{theorem}\label{T1}
Suppose $f(z)$ is a transcendental  meromorphic solution of 
\begin{equation}\label{E5}
{f}'(z)=a(z)f(z+1)+R(z,f(z)),\quad R(z,f(z))=\frac{P(z,f(z))}{Q(z,f(z))},    
\end{equation}
where $R(z,f(z))$ is a nonzero irreducible rational function in $f(z)$ with rational coefficients, $a(z)$ is a nonzero rational function.
If $f(z)$ is subnormal, then 
$$\deg_{f}(Q)=0, \deg_{f}(P)=0,1,2 \quad\text{or}\quad \deg_{f}(Q)=1, \deg_{f}(P)=3.$$
\end{theorem}

Under the assumptions of Theorem \ref{T1}, and based on its conclusions, there exist four possible forms for \eqref{E5}. 
If $\deg_{f}(P) - \deg_{f}(Q) = 2$, then \eqref{E5} takes the form:
\begin{align}\label{E10}
    &f'(z) = a(z) f(z+1) + R(z,f(z)),\\\nonumber 
&R(z,f(z))=b(z) f(z)^2 + c(z) f(z) + d(z),
\end{align}
or 
\begin{align}\label{E61}
    &f'(z) = a(z) f(z+1) + R(z,f(z)),\\\nonumber
&R(z,f(z))=\frac{a_0(z) f(z)^3 + a_1(z) f(z)^2 + a_2(z) f(z) + a_3(z)}{f(z) - a_4(z)},
\end{align}
where \(a(z)\), \(b(z)\), \(a_0(z)\), and \(a_3(z)\) are not identically zero. In this case, we obtain the following theorem. Additionally, in Theorem \ref{T3} we consider a more general case than Theorem \ref{T1}, where all coefficients of \(R(z,f(z)\)) are small functions of \(f(z\)).

\begin{theorem}\label{T3}
Suppose that \(f(z)\) is a transcendental meromorphic solution of \eqref{E10} or \eqref{E61} with subnormal growth, all coefficient functions of \(R(z,f(z))\) are small functions of \(f(z)\), and \(a(z)\) is also a small  function of $f(z)$. Then we have:

\begin{enumerate}
    \item[(i)] \( N_1(r, f(z)) = T(r, f(z)) + S(r, f(z)). \)

    \item[(ii)]  Equation \eqref{E10} reduces to a Riccati differential equation as follows:
    \begin{equation}\label{E396}
        f'(z) = b(z) f(z)^2 + \left[c(z) + \frac{a(z) b(z)}{b(z+1)}\right] f(z) + n(z),
    \end{equation}
    where \( T(r, n(z)) = S(r, f(z)) \).

    \item[(iii)] Equation \eqref{E61} reduces to a Riccati differential equation:
    \begin{equation}\label{E610}
        f'(z) + L_1(z) f(z) + L_2(z) f(z)^2 = L_3(z),
    \end{equation}
    where
    \begin{align*}
        L_1(z) &= \frac{a'_0(z) + 2a_0(z) a_4(z) - a_0(z) C(z)}{a_0(z)}, \\
        L_2(z) &= -a_0(z), \\
    \end{align*}
    and
  \begin{align}\label{N613}
 &T(r, L_{3}(z))=S(r,f(z)),\\\nonumber 
 &C(z)=\frac{{a_0}'(z)}{a_0(z)} + a_2(z) + 2a_1(z) a_4(z) + 3a_0(z) a_4(z)^2 + a(z) a_4(z+1) - {a_4}'(z).
 \end{align}
Alternatively, 
\begin{align*}
L_1(z) &=-\frac{a'_0(z)}{a_0(z)}+2a_4(z)-A(z)-C(z),\\\nonumber
L_2(z) &=-1,\\\nonumber
 \end{align*}
where   $A(z)=\frac{a(z)a_{0}(z)}{a_{0}(z+1)},$ $C(z)$ and $L_3(z)$ 
 satisfy \eqref{N613}.
\end{enumerate}
\end{theorem}

For the remaining case, i.e., when \(\deg_{f}(P) - \deg_{f}(Q) = 0, 1\), then \eqref{E5} takes the form \eqref{E8}, and we present the following theorem.


\begin{theorem}\label{T2}
Let \( f(z) \) be a transcendental meromorphic function that is a solution of the complex Schr\"{o}dinger equation with delay
\begin{equation}\label{E8}
    f'(z) = a(z) f(z+1) + b(z) f(z) + c(z),
\end{equation}
where \( a(z) \), \( b(z) \), and \( c(z) \) are rational functions , and \( a(z) \not\equiv 0 \). Then we have the following:

\begin{enumerate}
    \item[(i)] If all coefficient functions in \eqref{E8} are constants, then \eqref{E8} has only entire solutions. Additionally, if at least one of \( a(z) \), \( b(z) \), or \( c(z) \) is a non-constant rational function, then \( N(r, f(z)) = S(r, f(z)) \).
    
    \item[(ii)] Suppose  \( a(z) \not\equiv 1 \) and \( c(z) \not\equiv 0 \). If \( f(z) \) is a transcendental entire function solution of finite order of \eqref{E8}, then \( \sigma(f) \geq 1 \).
    
    \item[(iii)] Suppose \( a(z) \), \( b(z) \), and \( c(z) \) are complex constants, and that \( f(z) \) is of non-vanishing finite order.
    \begin{enumerate}
        \item[(a)] If \( a(z) + b(z) \not\equiv 0 \), \( a(z) \not\equiv 1 \), and \( c(z) \equiv 0 \), then \( \sigma(f) \geq 1 \).
        
        \item[(b)] If a complex constant \( k \) is a Borel exceptional value of \( f(z) \), and \( a(z) + b(z) \not\equiv 0 \), then \( k = \frac{-c}{a+b} \).
    \end{enumerate}
    
    \item[(iv)] If  \( c(z) \not\equiv 0 \), and if the exponential polynomial in \eqref{E690} is a transcendental entire function solution of \eqref{E8}, then \( Q_1(z), \dots, Q_k(z) \) have at least one term that reduces to a constant.
\end{enumerate}
\end{theorem}

\begin{remark}
\cite[Theorem 3.3]{bib32} is a special case of \eqref{E8} with $a(z)=1,$ $b(z)=0,$ $c(z)=0.$
\end{remark}
   
   With Lemma \ref{L21}, Remark \ref{r1}, and Theorem \ref{T2} (iv), we can immediately obtain following corollary.

\begin{corollary}\label{C1}
If \( a(z) \), \( b(z) \), and \( c(z) \) are rational functions with \( c(z) \not\equiv 0 \), and if the exponential polynomial in \eqref{E691} is a transcendental entire function solution of \eqref{E8}, then 
\begin{align*}
    m\left(r, \frac{1}{f}\right) = o(r^q) = S(r, f),
\end{align*} 
and
\begin{align*}
    N\left(r, \frac{1}{f}\right) = C(\operatorname{co}(W_0)) \frac{r^q}{2\pi} + o(r^q).
\end{align*}
\end{corollary}
\begin{remark}
The definition of exponential polynomial, as well as some notation appearing in Corollary  \ref{C1}, can be found in section $3.$    
\end{remark}

\section{Preliminaries}\label{sec3}
We assume that the reader is familiar with the standard notations and basic results of Nevanlinna theory. For a meromorphic function \( f \), the characteristic function is defined by
\[
T(r, f) = m(r, f) + N(r, f),
\]
where \( N(r, f) \) is the Nevanlinna counting function for the poles of \( f \), and \( m(r, f) \) is the Nevanlinna proximity function of \( f \). For a meromorphic function \( f \) in the complex plane, the order of \( f \) is defined as
\[
\sigma(f) = \limsup_{r \to \infty} \frac{\log T(r, f)}{\log r}.
\]
The exponent of convergence for the zero-sequence and pole-sequence of a meromorphic function \( f \) is given by
\[
\lambda(f) = \limsup_{r \to \infty} \frac{\log N\left(r, \frac{1}{f}\right)}{\log r},
\]
and
\[
\lambda\left(\frac{1}{f}\right) = \limsup_{r \to \infty} \frac{\log N(r, f)}{\log r}.
\]
If a complex constant \( a \) satisfies
\begin{equation}
\limsup_{r \to \infty} \frac{\log n(r, a)}{\log r} < \sigma(f),
\end{equation}
then \( a \) is called a Borel exceptional value of \( f(z) \).

Moreover, we say that a meromorphic function \( g \) is a small function with respect to a meromorphic function \( f \) if \( T(r, g) = S(r, f) \), where \( S(r, f) \) represents a quantity that satisfies \( S(r, f) = o(T(r, f)) \), as $r\to\infty$, possibly outside of an exceptional set of finite logarithmic measure or of zero upper density measure.

\begin{definition}
\( N_{1}(r, f) \) denotes the integrated counting function for simple poles of \( f \) in \( |z| < r \), and \( N_{[p}(r, f) \) denotes the integrated counting function for poles of order \( p \) or higher in \( |z| < r \). For \( a \in \mathbb{C} \), we similarly define \( N_{1}\left(r, \frac{1}{f - a}\right) \).
\end{definition}

\begin{definition} \cite{bib15, bib26}
A transcendental meromorphic function \( f(z) \) is subnormal if it satisfies
\[
\limsup_{r \to \infty} \frac{\log T(r, f)}{r} = 0,
\]
where \( T(r, f) \) is the Nevanlinna characteristic function of \( f(z) \).
\end{definition}

\begin{definition}\label{D3}
A differential-difference polynomial in \( f(z) \) is defined by
\begin{align*}
P(z, f) &= \sum_{l \in L} b_{l}(z) f(z)^{l_{0, 0}} f(z + c_{1})^{l_{1, 0}} \dots f(z + c_{\nu})^{l_{\nu, 0}} \left[f'(z)\right]^{l_{0, 1}} \dots \left[f^{(\mu)}(z + c_{\nu})\right]^{l_{\nu, \mu}},
\end{align*}
where \( c_{1}, \dots, c_{\nu} \) are distinct non-zero complex constants, \( L \) is a finite index set consisting of elements of the form \( l = (l_{0, 0}, \dots, l_{\nu, \mu}) \), and the coefficients \( b_l(z) \) are small functions of \( f(z) \) for all \( l \in L \).
\end{definition}

For more detailed definitions and notations of Nevanlinna theory, we refer the reader to \cite{bib21,bib27}.

We now define the concept of an exponential polynomial.  A function \( f(z) \) is called an exponential polynomial if it has the form
\begin{equation}\label{E690}
f(z) = P_{1}(z)e^{Q_{1}(z)} + \dots + P_{k}(z)e^{Q_{k}(z)},
\end{equation}
where \( P_j \) and \( Q_j \) are polynomials in \( z \). Let \( q = \max\{\deg(Q_j) : Q_j \not\equiv 0\} \), and let \( \omega_{1}, \dots, \omega_{m} \) be the distinct leading coefficients of the polynomials \( Q_j(z) \) of maximum degree \( q \). Then \eqref{E690} can be rewritten as
\begin{equation}\label{E691}
f(z) = H_{0}(z) + H_{1}(z)e^{\omega_{1} z^{q}} + \dots + H_{m}(z)e^{\omega_{m} z^{q}},
\end{equation}
where \( H_{j} \) are either exponential polynomials of degree less than \( q \) or ordinary polynomials in \( z \). By construction, we have \( H_{j}(z) \not\equiv 0 \) for \( 1 \le j \le m \).

\begin{definition}\label{D4} \cite{bib33,bib36}
The convex hull of a set \( W \subset \mathbb{C} \), denoted by \( \operatorname{co}(W) \), is the intersection of all convex sets containing \( W \). If \( W \) contains only finitely many elements, then \( \operatorname{co}(W) \) is obtained as an intersection of finitely many closed half-planes, and hence \( \operatorname{co}(W) \) is either a compact polygon (with a non-empty interior) or a line segment. We denote the perimeter of \( \operatorname{co}(W) \) by \( C(\operatorname{co}(W)) \). If \( \operatorname{co}(W) \) is a line segment, then \( C(\operatorname{co}(W)) \) is twice the length of this line segment. In this paper, we fix the notations \( W = \{\overline{\omega}_1, \dots, \overline{\omega}_m\} \) and \( W_0 = \{0, \overline{\omega}_1, \dots, \overline{\omega}_m\} \), concerning the constants \( \omega_{j} \) in equation \eqref{E691}.
\end{definition}

\section{Lemmas}\label{sec4}
The first lemma is the celebrated lemma of the logarithmic derivative.
\begin{lemma}\label{L1}
 Let $f(z)$ be a non-constant meromorphic function. Then 
\begin{flalign*}
m\left(r,\frac{{f}'(z) }{f(z)}\right)=O(\log rT(r, f)) \quad (r\longrightarrow\infty, r\notin E),
\end{flalign*}
where $E$ is of finite linear measure.
\end{lemma}

\begin{lemma}\label{L2}
\cite[Lemma 2.1]{bib42} Let \( T(r) \) be a non-decreasing, positive function on \( \left[1, +\infty\right) \), which is logarithmically convex and satisfies \( T(r) \to +\infty \) as \( r \to \infty \). Assume that 
\begin{equation*}
\liminf_{r \to \infty}\frac{\log T(r)}{r} = 0.
\end{equation*}
Define \( \phi(r) = \max_{1 \le t \le r} \left\{ \frac{t}{\log T(t)} \right\} \). Then, for a constant \( \delta \in \left(0, \frac{1}{2}\right) \), we have 
\begin{equation*}
T(r) \le T(r + \phi^{r}(r)) \le (1 + 4\phi^{\delta - \frac{1}{2}}(r)) T(r), \quad r \notin E_{\delta},
\end{equation*}
where \( E_{\delta} \) is a subset of \( \left[1, +\infty\right) \) with zero lower density. Furthermore, \( E_{\delta} \) has zero upper density if the condition (i) holds for \( \limsup \).
\end{lemma}

The difference version of the logarithm derivative lemma has been improved by Zheng and the second author \cite{bib42}  in 2020. This improvement extends the lemma from cases of finite growth order \cite{bib10, bib17} and hyper-order less than one \cite{bib18} to subnormal functions \cite{bib42}.

\begin{lemma}\label{L3}
\cite[difference version of logarithmic derivative lemma]{bib42} Let $f$ be a non-constant meromorphic function, and let $c\in \mathbb{C}$ with $c\ne0$. If 
\begin{flalign*}
\limsup_{r\to \infty }\frac{\log  T(r, f)}{r}=0, 
\end{flalign*}
then
\begin{equation*}
m\left ( r,\frac{f\left ( z+c \right ) }{f(z)} \right ) +m\left ( r,\frac{f(z)}{f(z+c)} \right ) =o(T(r, f))
\end{equation*}
for all $r\notin E$ where $E$ is a set with zero upper density measure $E$, i.e, $$\overline{{\rm dens}}E = \limsup_{r \to \infty} \frac{1}{r} \int_{E\cap \left [ 1, r \right ] }dt=0.$$
\end{lemma}

By Lemma \ref{L2} and Lemma \ref{L3}, it is easy to get the following equalities which will be used throughout this paper:
\begin{eqnarray*}
&&T(r, f(z+c))=T(r, f(z))+S(r, f(z)),\\
&&N(r, f(z+c))=N(r, f(z))+S(r, f(z)).
\end{eqnarray*}

The following lemma, due to Halburd and the second author \cite{bib16}, considers a transcendental meromorphic solution of the equation $P(z,w)=0$, where the coefficients are rational functions. Later, Xu and the first author \cite{bib39} further generalized this result to  admissible meromorphic solutions of the equation $P(z,f)=0$ by applying Lemma \ref{L3}.

\begin{lemma}\label{L4}
Let $f(z)$ be an admissible meromorphic solution of 
$$P(z,f(z))=0,$$
where $P(z,f(z))$ is a differential-difference polynomial, which is defined by Definition \eqref{D3}. Let $a_{1},\ldots,a_{k},$ be meromorphic functions small with respect to $f(z)$, such that $P(z,a_{j}(z))\not\equiv 0$  for all $j=1,\ldots,k$. If there exist $s> 0$ and $\tau \in \left ( 0, 1 \right )$ such that 
 \begin{flalign*}
 \sum_{j=1}^{k} n\left ( r, \frac{1}{f-a_{j} }  \right ) \le k\tau n(r+s, f)+O(1),
 \end{flalign*}
 then 
 $$\limsup_{r \to \infty} \frac{\log  T(r, f)}{r} >0.$$
 \end{lemma}

Laine and Yang provided Clunie theorems for differences when $f(z)$ is of finite order\cite[Theorem 2.3]{bib29}. Hu and Wang gave an improved Clunie theorem for differential-difference equations when the growth of $f(z)$ is of hyper-order strictly less than one\cite[Lemma 3.2]{bib23}. Indeed,
 we only need to replace   the difference version of the logarithmic derivative lemma with  $\sigma_2(f)<1$ by Lemma \ref{L3} in \cite[Lemma 3.2]{bib23} to obtain the following slightly improved lemma.

\begin{lemma}\label{L8}
Let $f(z)$ be a transcendental meromorphic solution of the differential-difference equation 
\begin{flalign*}
U(z,f(z))P(z,f(z))=Q(z,f(z)),
\end{flalign*}
where $U(z,f(z))$ is a difference polynomial in $f(z)$ with small meromorphic coefficients, $P(z,f)$ and $Q(z,f)$ are differential-difference polynomials in $f(z)$ with small meromorphic coefficients, $\deg_{f}(U)=n$ and  $\deg_{f} Q(z,f(z))\le n$. Moreover, assume that $U(z,f(z))$ contains only one term of maximal total degree in $f(z)$ and its shifts. If $$\limsup_{r \to \infty}\frac{\log T  (r,f)}{r}=0,$$ then we have
\begin{flalign*}
m(r,P(z,f(z))=S(r,f(z))
\end{flalign*}
outside an exceptional set of finite logarithmic measure or of zero upper density measure.
\end{lemma}

Now we present an improved lemma that extends  Mohon'ko's theorem and its difference analogue to differential-difference equations with small meromorphic coefficients, where the solutions exhibit subnormal growth. The proof of Lemma \ref{L10} closely follows the proof of \cite[Lemma 2.1]{bib16}.

\begin{lemma}\label{L10}
Let $f(z)$ be a transcendental meromorphic solution of subnormal growth of the equation
\begin{flalign}\label{E16}
P(z,f(z))=0,
\end{flalign}
where $P(z,f(z))$ is a differential-difference polynomial defined by Definition \eqref{D3}.
If $P(z,a(z))\not\equiv 0$ and $a(z)$ is a small function of $f(z)$, then 
\begin{align*}
m\left ( r,\frac{1}{f(z)-a(z)} \right ) =S(r,f(z)).
\end{align*}
\end{lemma}

\begin{lemma}\label{L21} \cite{bib36} Let $f$ be defined by \eqref{E691}. Then    \begin{align*}T(r,f)=C(\operatorname{co}(W_{0}))\frac{r^{q}}{2\pi} +o(r^{q}).\end{align*}
Additionally, if $H_0(z)\not\equiv 0,$ we have $$m\left ( r,\frac{1}{f} \right )=o(r^{q} ),$$
while if $H_{0}(z)\equiv 0,$ then $$N\left(r,\frac{1}{f}\right)=C(\operatorname{co}(W)) \frac{r^{q}}{2\pi} +o(r^{q}).$$ 
\end{lemma}

\begin{remark}\label{r1}
\cite[remark 5]{bib33} Suppose that $f$ is an exponential polynomial  of form \eqref{E691} with $m\ge 1$. Then, we have following fact: 
\begin{align*}
S(r,f)=o(r^{q}), o(r^{q})= S(r,f).  
\end{align*}
\end{remark}

\section{Proof of Theorem \ref{T1}}\label{sec5}
We start with two important lemmas that are used to prove  
Theorem \ref{T1}.
\begin{lemma}\label{L11}
Let $f(z)$ be a transcendental meromorphic solution of \eqref{E5}, where all coefficient functions of $R(z,f(z))$ and $a(z)$ are small functions of $f(z).$ If $f(z)$ is of subnormal growth, then we have $$N\left ( r,\frac{1}{Q(z,f(z))}\right)\le N(r,f(z))+S(r,f(z)).$$ 
\end{lemma}

\begin{proof}
Define \begin{align}\label{E795} 
A=\left \{ z:f(z)=\infty  \right \},B=\left \{ z:f(z)\not =\infty\right \}. 
\end{align}
By \eqref{E5}, we have
\begin{align}\label{E796}
&N(r,R(z,f(z)))\\\nonumber
=&\max\left \{\deg_{f}(P)-\deg_{f}(Q),0\right\}N(r,f(z))+N\left ( r,\frac{1}{Q(z,f(z))}\right)+S(r,f(z)).
\end{align}  
Using \eqref{E795} and \eqref{E796}, it follows that \begin{align}\label{6}
N_{B}(r,R(z,f(z)))=N\left ( r,\frac{1}{Q(z,f(z))}\right)+S(r,f(z)).
\end{align} 
On the other hand, by applying  \eqref{E795}  and Lemma \ref{L2}  to \eqref{E5}, we have 
\begin{align}\label{E98}
N_{B}(r,R(z,f(z)))=& N_{B}(r,{f}'(z)-a(z)f(z+1))\\\nonumber
\le& N_{B}(r,{f}'(z))+N_{B}(r,{f}(z+1)) +S(r,f(z))\\\nonumber  
\le& N(r,f(z+1))+S(r,f(z))\\\nonumber
=&N(r,f(z))+S(r,f(z)).
\end{align}
By combining \eqref{6} and \eqref{E98}, we obtain $$N\left ( r,\frac{1}{Q(z,f(z))}\right)\le N(r,f(z))+S(r,f(z)).$$ 
\end{proof}

\begin{remark}
By applying  Lemma 
\ref{L10} to \eqref{E15}, we have \begin{align}\label{E102}
N\left(r,\frac{1}{f(z)-b_1(z)}\right) & = T(r,f(z))+S(r,f(z)).
\end{align} 
Note also that all coefficient functions of \eqref{E15} are rational functions. Hence, we can assume that $z=\hat{z} $ is zero of $f(z)-b_1(z),$ but $z=\hat{z} $ is neither a zero or a pole of any coefficient function of \eqref{E15}. 
We say that $z_{0}$ is a $\mathit{generic}$ $\mathit{zero}$ of $f(z)-b_1(z)$ when above requirements are satisfied. Throughout the rest of this paper, we similarly define the terms generic zero and generic pole of $f(z).$\par
We cannot deduce by \eqref{E102} that the generic zero of $f(z)-b_1(z)$  must exist if all coefficients of \eqref{E15} are small functions that are non-rational of \(z\). Indeed, there exist the case where \(z=z_0\) is an arbitrary  zero of order \(t\) of 
\(f(z)-b_1(z)\), and is zero or pole of order \(t_1\) of arbitrary coefficient of \eqref{E15},  \(t>t_1\).
\end{remark}

\begin{lemma}\label{L7}
Let $f(z)$ be a transcendental meromorphic solution of
\begin{align}\label{E15}
f'(z)=a(z)f(z+1)+R(z,f(z)), R(z,f(z))=
\frac{P(z,f(z))}{(f(z)-b_{1}(z))^{k}\widetilde{Q} (z,f(z))}. 
\end{align}
where $a(z)$ is a nonzero rational function,
$R(z,f(z))$ is rational in both arguments, and all coefficient functions of $R(z,f(z))$ are rational functions. Let
$k$ be an integer greater than $1.$ Furthermore, assume that \( P(z, f(z)) \), \( f(z) - b_{1}(z) \), and \( \widetilde{Q}(z, f(z)) \) are pairwise relatively prime. Then 
\begin{align*}
\limsup_{r\to \infty}\frac{\log  T(r,f)}{r}>0.
\end{align*}
\end{lemma}

\begin{proof}
Clearly, $b_{1}(z)$ is not a solution of equation \eqref{E15}, even if $b_{1}(z)\equiv 0$. Suppose that $f(z)-b_{1}(z)$ has a generic zero of order $t$ at $z=\hat{z}$ (the existence of the generic zeros of $f(z)-b_{1}$ can be established by Lemma \ref{L10} and the fact that that all coefficients of \eqref{E15} are rational). 
Then $f(z+1)$ has a pole of order $tk$ at $z=\hat{z}.$

If $\deg_{f}(P)>k+\deg_{f} (\widetilde{Q} ),$ shifting \eqref{E15} by one step yields that  $f(z+2)$ has a pole of order at least $tk+1$ at $z=\hat{z}.$ Continuing to shift \eqref{E15} further, we obtain an infinite iterative sequence $\left \{ f(\hat{z}+n) \right \}_{n=1}^{\infty},$ where  $f(\hat{z}+n)=\infty $ for all  positive integers $n.$ In addition,  $f(z+n)$ has a pole of order at least $tk+n-1$ at $z=\hat{z}.$ 

Fix a positive integer $m$ $(\ge 2).$
By computing the ratio of zeros of $f(z)-b_1(z)$ to the poles of $f(z)$ in the set $\left \{ \hat{z}, \hat{z}+1,  \hat{z}+2,\dots , \hat{z}+m \right \},$ we obtain 
\begin{equation}\label{E701}
n\left ( r,\frac{1}{f(z)-b_{1}(z) } \right ) \le \frac{1}{m} n(r+m,f(z))+O(1).
\end{equation}

Similarly, if $\deg_{f}(P)\le k+\deg_{f}(\tilde{Q})$, we find  that \eqref{E701} is still valid.
Applying  Lemma \ref{L4} to \eqref{E701}, we conclude that 
$\limsup_{r\to \infty}\frac{{\rm log}T (r,f)}{r}>0.$\end{proof}

\begin{proof}[Proof of Theorem 2.1]

\textbf{Step 1:} We first  assert that $\deg_{f}(Q)\le 1.$ Suppose, on the contrary, that 
   $\deg_{f}(Q)\ge 2.$ 
Applying Lemma \ref{L7} to \eqref{E5}, we obtain
  
\begin{align}\label{E91}
{f}'(z)=a(z)f(z+1)+\frac{P(z,f(z))}{(f(z)-b_{1}(z))(f(z)-b_{2}(z))\hat{Q}(z,f(z))},  
\end{align}
where either $\deg_{f}(\hat{Q})=0$ or
$\hat{Q}(z,f(z))$ is polynomial with respect to $f(z)$ and $z,$ with rational coefficients, and $b_{1}(z)\not\equiv b_{2}(z).$   

Using Lemma \ref{L10}, we find that  \begin{align}\label{E92}
N\left ( r,\frac{1}{f(z)-b_{1}(z)}\right)=T(r,f(z))+S(r,f(z)), 
\end{align}
and \begin{align}\label{E93}
N\left ( r,\frac{1}{f(z)-b_{2}(z)}\right)=T(r,f(z))+S(r,f(z)). 
\end{align}
Note also that all coefficient functions of \eqref{E5} are rational functions. Hence the  generic zeros of $f(z)-b_{j}(z)$ $(1\le j\le 2)$ must exist (the existence of the generic zeros of $f(z)-b_{j}$ $(1\leq j \leq 2)$ can be established by Lemma \ref{L10} and the fact that that all coefficients of \eqref{E5} are rational).

Now assume that $z_{0}$ is a generic zero of $f(z)-b_1(z)$ or $f(z)-b_2(z).$
We assert that there  does not exist $z_{0}\in \mathbb{C}$ such that either $f(z_{0})-b_{1}(z_{0})=0, f(z_{0})-b_{2}(z_{0})=\infty$ or $f(z_{0})-b_{1}(z_{0})=\infty , f(z_{0})-b_{2}(z_{0})=0.$ Otherwise, we get $b_{1}(z_{0})-b_{2}(z_{0})=\infty,$ which implies that $z=z_{0}$ is a pole of $b_{1}(z)$ or $b_{2}(z).$ This is a contradiction with $z_{0} $ being a generic zero of $f(z)-b_1(z)$ or of $f(z)-b_2(z)$.  

Thus, we obtain the  following:\begin{align}\label{E94}
&N\left ( r, \frac{1}{Q(z,f(z))} \right )\\\nonumber 
=& N\left ( r,\frac{1}{(f(z)-b_{1}(z))(f(z)-b_{2}(z))\hat{Q}(z,f(z))}\right)\\\nonumber
\ge& N\left ( r,\frac{1}{(f(z)-b_{1}(z))(f(z)-b_{2}(z))}\right)+S(r,f(z))\\\nonumber
=&N\left (r,\frac{1}{f(z)-b_{1}(z)}\right)+N\left (r,\frac{1}{f(z)-b_{1}(z)}\right)+S(r,f(z)).\\\nonumber  
\end{align}
Combining \eqref{E92},  \eqref{E93}, and \eqref{E94} results in \begin{align}\label{E95}
N\left (r,\frac{1}{Q(z,f(z))}\right) \ge 2T(r,f(z))+S(r,f(z)).
\end{align}
On the other hand, \eqref{E95} contradicts Lemma \ref{L11}. Therefore, we conclude that  
$\deg_{f}(Q)\le 1$ is proved.

\textbf{Step 2:}
Next, we give an estimate for $\deg_{f}(R).$  
By
taking the characteristic function of the both sides of \eqref{E5} and applying Lemma \ref{L2}, \cite[Theorem 2.2.5.]{bib27}, we get 
\begin{align}\label{E694} 
\max\left \{ \deg_{f}(P),\deg_{f}(Q)\right \}T(r,f)
=&T(r,R(z,f(z)))\\\nonumber
=&T(r,{f}'(z)-a(z)f(z+1))\\\nonumber
\le& T(r,{f}'(z))+T(r,f(z+1))+S(r,f(z))\\\nonumber
\le& 3T(r,f(z))+S(r,f(z)),
\end{align}
which  immediately yields 
$$\max \left \{ \deg_{f}(P),\deg_{f}(Q)\right \}\le 3.$$

\textbf{Step 3:} By combining step 1 and step 2, we find that
$\deg_{f}(P)\le 3$ and $\deg_{f}(Q)\le 1    .$ We now classify cases based on $\deg_{f}(P)$ and $\deg_{f}(Q).$ 

\textbf{Case 1.} We show that $\deg_{f}(P)=3$ and $\deg_{f}(Q)=0$ cannot occur. Suppose, on the contrary, that $\deg_{f}(P)=3$ and $\deg_{f}(Q)=0$. Then \eqref{E5} reduce to \begin{eqnarray}\label{E87}
{f}'(z)=a(z)f(z+1)+l_{0}(z)f(z)^{3}+l_{1}(z)f(z)^{2}+l_{2}(z)f(z)+l_{3}(z), 
\end{eqnarray}
where $P(z,f(z))=l_{0}(z)f(z)^{3}+l_{1}(z)f(z)^{2}+l_{2}(z)f(z)+l_{3}(z),$ and
$l_{0}(z)\not\equiv0.$  

Suppose $z_{0}$ is arbitrary generic pole of order $t$ of $f(z),$ i.e., $f(z_{0})=\infty ^{k}.$ Then ${f}'(z_{0})=\infty ^{k+1},$   $P(z_{0},f(z_{0}))=\infty ^{3k},$ and   $f(z_{0}+1)=\infty ^{3k}$. Since $z_{0}$ was chosen as an arbitrary generic pole of order $t$ of $f(z)$, applying Lemma \ref{L2} gives
\begin{align}\label{8}
3N(r,f(z))+S(r,f(z))=N(r,P(z,f(z)))=&N(r,f(z+1))+S(r,f(z))\\\nonumber
=&N(r,f(z))+S(r,f(z)),
\end{align}
which implies\begin{align}\label{E88}
N(r,f(z))=S(r,f(z)).
\end{align}  Suppose there is no generic poles of $f(z),$ then we still have \eqref{E88}.\par

Note that \eqref{E87} can be written as 
\begin{align}\label{E89}
{f}'(z)-a(z)f(z+1)-l_{3}(z)=f(z)(l_{0}(z)f(z)^{2}+l_{1}(z)f(z)+l_{2}(z)).  
\end{align}
Applying Lemma \ref{L8} to \eqref{E89}, we get \begin{align}\label{E90}
m(r,f(z))=S(r,f(z)).
\end{align}
Combining \eqref{E88} and \eqref{E90}, it follows that 
$T(r,f(z))=S(r,f(z)),$ which yields a contradiction. The proof of Case $1$ is  complete.

\textbf{Case 2.}
We assert  that $\deg_{f}(P)=0,\deg_{f}(Q)=1$ cannot occur. Suppose, on the contrary, that  $\deg_{f}(P)=0$ and $\deg_{f}(Q)=1.$
Then \eqref{E5} becomes \begin{align}\label{E390}
{f}'(z)=a(z)f(z+1)+\frac{a_{0}(z)}{f(z)-b_{1}(z)},\quad a_{0}(z)\not\equiv 0.    
\end{align}
Let $z=z_0$ be a generic zero of $f(z)-b_1(z)$ with multiplicity $t.$ Then $f(z_0+1)=\infty^{t}.$ Shifting \eqref{E390} up gives $f(z_0+2)=\infty ^{t+1}.$ Repeating this, 
we obtain an infinite iterative sequence $\left \{ f(z_{0}+n) \right \}_{n=1}^{\infty} $ where $f(z_{0}+n)=\infty $ for all positive integers $n.$ 

Fix a positive inter $m$. It is easy to verify that $$\frac{n\left ( r,\frac{1}{f(z)-b_1(z)} \right)}{n(r+m,f(z))}\le \frac{1}{m}.$$
Applying Lemma \ref{L4},
this contradicts  $$\limsup_{r \to \infty}\frac{\log T(r,f)}{r}=0.$$
Therefore, our assertion  is  confirmed.

\textbf{Case 3.}
Assume now that the cases 
$\deg_{f}(P)=1$ and $\deg_{f}(Q)=1,$ as well as $\deg_{f}(P)=2$ and $\deg_{f}(Q)=1,$ can occur. These situations are very similar to Case $2.$ We only need to consider the ratio of the zeros of $f(z)-b_1(z)$ to the poles of $f(z)$ and apply Lemma \ref{L4}  to obtain contradictions with $f(z)$ having subnormal growth. 
Details are omitted. 

Combining Step $1$, $2$, and $3$ completes the proof of  Theorem \ref{T1}. 
\end{proof}

\section{Proof of Theorem \ref{T3}}\label{sec6}
\begin{proof}
We present a crucial lemma that is utilized to prove Theorem \ref{T3}.
\begin{lemma}\cite[Lemma 3.1]{bib19}\label{L14}
Let $f(z)$ be an admissible meromorphic solution of \eqref{E5}   
 with more than $S(r,f(z))$ poles (counting multiplicities) and let $z_{j}$ denote the zeros and poles of the coefficients $a_{i}$ of $R(z,f(z)).$ Let  \begin{align*}
 m_{j}:=\underset{i=1,\dots,n}{\max}\left\{l_{i}\in \mathbb{N}:a_{i}(z_{j})=0^{l_{i}},\text{or}, a_{i}(z_{j})=\infty^{l_{i}} \right\}     
\end{align*}
be the maximal order of zeros and poles of the functions $a_{i}$ at $z_{j}.$ Then for any $\epsilon>0$ there are at most $S(r,f)$ points $z_{j}$ such that \begin{align}\label{N56}
f(z_{j}) & = \infty^{k_{j}} 
\end{align}
where $m_{j}\ge \epsilon k_{j}. $
\end{lemma}

(i). If $\deg_{f}(P)-\deg_{f}(Q)=2,$ then \eqref{E10} and \eqref{E61} can respectively take the following forms:
\begin{align}\label{E800} 
f(z)(b(z)f(z)+c(z))={f}'(z)-a(z)f(z+1)-d(z), 
\end{align}
\begin{align}\label{E801}
 &f(z)(a_{0}(z)f(z)^{2}+a_{1}(z)f(z)+a_{2}(z))\\\nonumber
=&\left [ ({f}'(z)-a(z)f(z+1))(f(z)-a_{4}(z))\right]-a_{3}(z).     
\end{align}
Using Lemma \ref{L8} to \eqref{E800} and \eqref{E801}, we see that \begin{align}\label{N20} 
N(r,f(z))=T(r,f(z))+S(r,f(z)),   
\end{align} that is to say, \eqref{N20} is valid
for \eqref{E10} and \eqref{E61}.

We now consider \eqref{E10}, \eqref{E61} and assert that the generic poles of \(f(z)\) exist in \eqref{E10}, \eqref{E61}, and 
$$N_{1}(r,f)=T(r,f)+S(r,f).$$  We firstly consider \eqref{E10}. We discuss two cases: 

\textbf{Case 1:} Suppose that $a(z)$ is rational function.

\textbf{Subcase 1.1:} All coefficients of $R(z,f(z))$ are rational functions. Since $f(z)$ is a transcendental meromorphic function and, as observed from \eqref{N20}, $f(z)$ has infinitely many poles, the generic poles of $f(z)$ must therefore exist in \eqref{E10}. We now assume that  $f(z_{0})=\infty ^{k}$ ($k>1$), where $z=z_{0}$ is an arbitrary generic pole of $f(z).$
Then ${f}'(z_{0})=\infty ^{k+1}$  and $R(z_{0},f(z_{0}))=\infty ^{2k}$ (since $\deg_{f}(P)-\deg_{f}(Q)=2$) in \eqref{E5}.
By \eqref{E10} we obtain $f(z_{0} +1)=\infty ^{2k}.$ 
By shifting \eqref{E5} up, we have
\begin{eqnarray}\label{11}
{f}'(z+1)=a(z+1)f(z+2)+R(z+1,f(z+1)),     
\end{eqnarray}
and by comparing the number of poles of both sides from identity \eqref{11}, it follows that   ${f}'(z_{0} +1)=\infty ^{2k+1}$, $R(z_{0} +1,f(z_{0} +1))=\infty ^{4k},$ $f(z_{0} +2)=\infty ^{4k}.$ By shifting \eqref{11} one more step, we get 
\begin{eqnarray*}
{f}'(z+2)=a(z+2)f(z+3)+R(z+2,f(z+2)).     
\end{eqnarray*}
Then ${f}'(z_{0} +2)=\infty ^{4k+1}$, $R(z_{0} +2,f(z_{1} +2))=\infty ^{8k},$ $f(z_{0} +3)=\infty ^{8k}.$ 
By repeating this process, we obtain $f(z_{0}+n)=\infty^{2^{n}k}$ for arbitrary positive $n,$ which implies that the growth of  $f(z)$ is not  subnormal.
Indeed, from \cite[p.97]{bib9} we have $$n(s,f)\leq \frac{r}{r-s}N(r,f)+O(1)\leq \frac{r}{r-s}T(r,f)+O(1)$$
when $r>s.$ Let $r=2\left|z_{0}\right| +2n, s=\frac{r}{2}=\left|z_{0}\right|+n,$ then we have following:
\begin{align}\label{N46}
\limsup_{r \to \infty}\frac{\log T(r,f)}{r}&\geq \limsup_{n \to \infty}\frac{\log\frac{1}{2}n(n+\left | z_{1} \right | ,f)}{2\left|z_{1}\right|+2n}\\\nonumber  
&\geq \limsup_{n \to \infty}\frac{\log \frac{1}{2}\left(2^{n}k \right)}{2\left|z_{1}\right|+2n}\\\nonumber 
&=\frac{\log 2}{2}\\\nonumber
&>0, 
\end{align}
which implies the growth of $f(z)$ does not satisfy the definition of subnormal. Thus, $k=1.$

\textbf{Subcase 1.2:} Not all coefficients of $R(z,f(z))$ are rational functions. In this case, suppose, on the contrary, that  \(f(z)\) has no generic poles.  Then if \(f(z_0)=\infty \), we have \(k(z_0)=0\) or \(k(z_0)=\infty \), where \(k(z)\) is an arbitrary coefficient of \eqref{E10}. Note that since all coefficients of \eqref{E10} are small functions of $f(z),$ it follows by  combining \eqref{N20} with Lemma \ref{L14}, that the order of the poles of $f(z)$ at $z=z_0$ is in general much greater than the order of the zeros or poles of $k(z)$ at $z=z_0$ and the error term is at most $S(r,f),$
  where $z=z_0$ is arbitrary pole of $f(z).$ In precise terms, there exists a set \begin{align}\label{N55}
 M=\left \{z_{0}: f(z_0)=\infty^{t_{0}}, k(z_{0})=0^{m},or, k(z_{0})=\infty ^{m},t_{0}\gg m  \right \}   \end{align}
 such that
\begin{align}\label{N31}
N_{M}(r,f(z))=T(r,f(z))+S(r,f(z)),   \end{align}  
where $t_{0}\gg m$
indicates that  for each fixed $\epsilon\in (0,1),$ there exists $z_{0}\in M$  such that $\epsilon t_{0}> m$  
in \eqref{N55}, and the term $N_{M}(r,f(z))$ refers to the counting function for the poles of $f(z)$ that belong to the set $M.$

Namely, by \eqref{E10} we have following:\begin{align}\label{N32}
 &f(z_0)=\infty^{t_{0}}, c(z_0)f(z_0)=\infty^{t_0+\epsilon_{2}}, f'(z_0)=\infty^{t_{0}+1}, b(z_0)f(z_0)^{2}=\infty^{2t_0+\epsilon_{1}},\\\nonumber  
 &d(z_{0})=\infty^{\epsilon_{3}},  
\end{align}
where if $\epsilon_{j}\ne 0$ $(j\in \mathbb{N}),$ $\epsilon_{1}$ denotes the order of the zeros or poles of $b(z)$ at $z=z_0,$ 
$\epsilon_{2}$ denotes the order of the zeros or poles of $c(z)$ at $z=z_0,$ and $\epsilon_{3}$ denotes the order of the zeros or poles of $d(z)$ at $z=z_0.$ 
Furthermore,  $\epsilon_{j}>0$ signifies the order of pole, while $\epsilon_{j}<0$ signifies the order of zero. Additionally, $\epsilon_{j}=0$ indicates that the coefficients are finite and nonzero at $z=z_{0}.$
Then by \eqref{N55} and \eqref{N31}, we have
\(\left|\epsilon_{1} \right|\), \(\left|\epsilon_{2} \right|\) and
\(\left|\epsilon_{2} \right|\) are given non-negative integers that satisfy  $t_{0}\gg\max\left\{\left|\epsilon_{1}\right|,\left|\epsilon_{2}\right|,\left|\epsilon_{3}\right|\right\},$  i.e., \begin{align}\label{E41} 
2t_{0}+\epsilon_{1}>\max\left\{t_{0}+1, t_{0}+\epsilon_{2}, t_{0},\epsilon_{3}\right\}.
\end{align}  In other words, 
 the zeros or poles of \(k(z)\) have slight influence on the order of pole of \(f(z)\).
Since \(a(z)\) is rational, all of its zeros and poles lie within a closed circle. Additionally, note  that $f(z)$ is a transcendental meromorphic function and \eqref{N20}. 
Then we can select a sequence $\left\{z_{0},z_{0}+1,z_{0}+2,\dots, \right \}$ such that the sequence lies outside of a sufficiently large circle, where $z=z_0$ is a pole of $f(z).$ Therefore, the zeros or poles of $a(z)$ do not influence the order of $f(z)$ if we select $z=z_{0}$ as the starting point for the iteration.
Hence by \eqref{E10}, \eqref{N32} and \eqref{E41}, we have  \(b(z_{0})f(z_0)^{2}=a(z_0)f(z_{0}+1)=f(z_{0}+1)=\infty^{2 t_{0}+\epsilon_{1}}\). By shifting \eqref{E10} up, 
we have 
\begin{align}\label{N33}
{f}'(z+1)=a(z+1)f(z+2)+b(z+1)f(z+1)^{2}+c(z+1)f(z+1)+d(z+1).    \end{align}
 Then \begin{align}\label{N34}
&f{'}(z_{0}+1)=\infty^{2t_0+\epsilon_{1}+1},b(z_{0}+1)f(z_{0}+1)^{2}=\infty^{4t_{0}+2\epsilon_{1}+\epsilon_{4}},\\\nonumber 
&c(z_{0}+1)f(z_{0}+1)=\infty^{2t_0+\epsilon_{1}+\epsilon_{5}}, d(z_{0}+1)=\infty^{\epsilon_{6}},  
\end{align}
where  $\epsilon_{4}$ denotes the order of the zeros or poles of $b(z+1)$ at $z=z_0,$ $\epsilon_{5}$ denotes the order of the zeros or poles of $c(z+1)$ at $z=z_0,$ $\epsilon_{6}$ denotes the order of the zeros or poles of $d(z+1)$ at $z=z_0.$ 
By Combining \eqref{N33} with \eqref{N34}, we have \begin{align}\label{N44}
f(z_{0}+2)=a(z_{0}+1)f(z_{0}+2)=b(z_{0}+1)f(z_{0}+1)^{2}=\infty^{4t_{0}+2\epsilon _{1}+\epsilon_{3}}.
\end{align}  
Since $\epsilon_{1},  \epsilon_{2}, \epsilon_{3}, \epsilon_{4}, 
\epsilon_{5},$
and $\epsilon_{6}$ 
are much smaller than $t_{0}.$ 
Shifting \eqref{N33} up, we have 
\begin{align}\label{N43}
{f}'(z+2)=a(z+2)f(z+3)+b(z+2)f(z+2)^{2}+c(z+2)f(z+2)+d(z+2).    
\end{align}
Then by \eqref{N44} we have following:
\begin{align}
&f{'}(z_{0}+2)=\infty^{4t_{0}+2\epsilon_{1}+\epsilon_{3}+1}, b(z_{0}+2)f(z_{0}+2)^{2}= \infty^{8t_{0}+4\epsilon_{1}+2\epsilon_{3}+\epsilon_{7}}\\\nonumber
&c(z_{0}+2)f(z_{0}+2)=\infty^{4t_{0}+2\epsilon_{1}+\epsilon_{3}+\epsilon_{8}},d(z_{0}+2)=\infty^{\epsilon_{9}},\\\nonumber      
\end{align}
where $\epsilon_{7}$ denotes the order of the zeros or poles of $b(z+2)$ at $z=z_0,$ $\epsilon_{8}$ denotes the order of the zeros or poles of $c(z+2)$ at $z=z_0,$ $\epsilon_{9}$ denotes the order of the zeros or poles of $d(z+2)$ at $z=z_0,$ 
Without loss of generality, we assume that $z=z_{0}$ is a simple zero of 
$R(z,f(z))$ and its shift, specifically, $\epsilon_{j}=-1$ $(j=1,2,\dots).$ If $\epsilon_{j}$ $(j=1,2,\dots)$ attain $0$ or a positive integer constant, then the order of pole of $f(z_{0}+n)$ at $z=z_{0}$ is higher.
Then \begin{align}\label{N45}
&f(z_{0})=\infty^{t_{0}}, f(z_{0}+1)=\infty^{2t_{0}-1}, f(z_{0}+2)=\infty^{4t_{0}-3},\\\nonumber
&a(z_{0}+2)f(z_{0}+3)=f(z_{0}+3)=b(z_{0}+2)f(z_{0}+2)^{2}= \infty^{8t_{0}+4\epsilon_{1}+2\epsilon_{3}+\epsilon_{7}}= \infty^{8t_{0}-7},
\end{align}
and by repeating the process above, we can eventually obtain that \begin{align}\label{N53}
f(z_{0}+n)=\infty^{2^{n}t_{0}-2^{n}+1}=\infty^{2^{n}(t_{0}-1)+1}.
\end{align}
By \eqref{N55} and \eqref{N31} we find that $t_{0}\gg 1.$ For \eqref{N53}
similar to the computation in \eqref{N46}, we can deduce a contradiction with $f(z)$ is subnormal growth. Thus, we assertion is proved, i.e., \(f(z)\) has generic poles for \eqref{E10}. 

Next, we prove every generic pole of $f(z)$ is simple. Suppose, for the sake of contradiction, that $f(z_{0})=\infty^{t_{0}},$ where $z=z_{0}$ is any generic pole of $f(z)$ and $t_{0}>1.$
From \eqref{E10} we have \begin{align}\label{N47}
&f(z_0)=\infty^{t_{0}}, c(z_0)f(z_0)=\infty^{t_0}, f'(z_0)=\infty^{t_{0}+1}, b(z_0)f(z_0)^{2}=\infty^{2t_0}.\\\nonumber   
&f(z_{0}+1)=b(z_0)f(z_0)^{2}=\infty^{2t_0}.
\end{align}
In this case, $a(z)$ is rational function, while $b(z),$ $c(z)$ and $d(z)$ are small functions of $f(z),$ which can be transcendental meromorphic functions. This implies that it is possible that if $z=z_{0}$ is a generic pole of $f(z)$ and if $z=z_{0}+1$ is also a pole of $f(z),$ then $z=z_{0}+1$ can be a non-generic pole of $f(z).$ Therefore, based on previous proof , we conclude that if we choose $z=z_0$ as the starting point for the iteration, then the zeros or poles of $a(z)$ do not cancel the poles of $f(z).$ However, we must also consider the zeros or poles of the coefficients when shifting \eqref{E10} upward.
By shifting \eqref{E10}, we have following:
\begin{align*}
{f}'(z+1)=a(z+1)f(z+2)+b(z+1)f(z+1)^{2}+c(z+1)f(z+1)+d(z+1).   
\end{align*}
 Then from \eqref{N47}, we have \begin{align}\label{N49}
&f{'}(z_{0}+1)=\infty^{2t_0+1}, c(z_{0}+1)f(z_{0}+1)=\infty^{2t_0+\varepsilon_{2}}, d(z_{0}+1)=\infty^{\varepsilon_{3}},\\\nonumber 
&f(z_{0}+2)=a(z_{0}+1)f(z_{0}+2)=b(z_{0}+1)f(z_{0}+1)^{2}=\infty^{4t_{0}+\varepsilon _{1}}.
\end{align}
By shifting \eqref{E10} upward again, we have  
\begin{align*}
{f}'(z+2)=a(z+2)f(z+3)+b(z+2)f(z+2)^{2}+c(z+2)f(z+2)+d(z+2).   
\end{align*}
By \eqref{N49} we obtain that \begin{align}\label{N50}
&f{'}(z_0+2)=\infty^{4t_{0}+\varepsilon_{1}+1}, b(z_{0}+2)f(z_{0}+2)^{2}=\infty^{8t_{0}+2\varepsilon_{1}+\varepsilon_{4}},\\\nonumber    
&c(z_{0}+2)f(z_{0}+2)=\infty^{4t_{0}+\varepsilon_{1}+\varepsilon_{5}}, d(z_{0}+2)=\infty^{\varepsilon_{6}},\\\nonumber
&f(z_{0}+3)=a(z_{0}+2)f(z_{0}+3)=b(z_{0}+2)f(z_{0}+2)^{2}=\infty^{8t_{0}+2\varepsilon_{1}+\varepsilon_{4}}.
\end{align}
As previously stated, the definitions of $\varepsilon_{j}$ $(j=1,\dots ,6)$ can refer \eqref{N32}.
By shifting \eqref{E10} more steps and note also that $\varepsilon_{j}$ $(j=1,\dots,6)$ is much smaller than $t_{0}.$ Without loss of generality, we assume that $\varepsilon_{j}=-1$ $(j=1,\dots).$ Then
\begin{align*}
&f(z_{0}+2)=\infty^{4t_{0}-1},f(z_{0}+3)=\infty^{8t_{0}-3}, \\\nonumber
&f(z_{0}+4)=\infty^{16t_{0}-7},\dots f(z_{0}+n)=\infty^{2^{n}t_{0}-2^{(n-1)}+1}.
\end{align*}
Similarly to \eqref{N46}, we can conclude that $f(z)$ does not satisfy the definition of subnormal growth.
Therefore, generic poles of $f(z)$ must exist.

Additionally, if $z=z_{0}$ is arbitrary generic pole of $f(z),$ suppose that $f(z_0)=\infty^{t}$ $(t>1).$ By \eqref{E10} we have $f(z_0+1)=\infty^{2t}.$ Shifting \eqref{E10} up, we have
\begin{equation}\label{N70}
 {f}' (z+1)=a(z+1)f(z+2)+b(z+1)f(z+1)^{2}+c(z+1)f(z+1)+d(z+1),    
\end{equation}
Then $f(z_{0}+2)=\infty^{4t}.$ By continuing to shift \eqref{N70} up, \begin{equation}\label{N71}
 {f}' (z+2)=a(z+2)f(z+3)+b(z+2)f(z+2)^{2}+c(z+2)f(z+2)+d(z+2),    
\end{equation}
we obtain that $f(z_{0}+3)=\infty^{8t}.$ If we repeat the process, it follows that $f(z_{0}+n)=\infty^{2^{n}} 
.$ By  \eqref{N46}, we have $f(z)$ does not exhibit  subnormal growth. Hence $t=1,$ i.e., all generic poles of $f(z)$ are simple poles. In the following case $2$, we can use exactly the same approach to prove that if the generic poles of $f(z)$ do exist then almost all generic poles of $f(z)$ must be simple and there are at most $S(r,f)$ points $z_{1}$ such that $f(z_{1})=\infty^{t}$ ($t\geq 2).$

\textbf{Case 2:} Suppose that $a(z)$ is non-rational meromorphic function.

\textbf{Subcase 2.1:} Assume that all coefficients of $R(z,f(z))$ are rational functions. Suppose, for the sake of contradiction, that $f(z)$ does not have any generic poles. Let $f(z_{0})=\infty^{t_{0}},$ where $t_{0}$ is a positive integer. From the proof of Subcase $1.2$, we find that we can choose a sequence $\left\{z_{0},z_{0}+1,z_{0}+2,\dots, \right \}$ such that this sequence lies outside of a sufficiently large circle, where $z=z_0$ is a pole of $f(z).$ In this manner, the zeros or poles of the coefficients of $R(z,f(z))$ cannot cancel the poles of $f(z)$ if we select $z=z_{0}$ as the starting point for the iteration. However, we must also consider the zeros or poles of $a(z).$
Thus, by \eqref{E10} we have following:
\begin{align}\label{N48}
&f(z_{0})=\infty^{t_{0}}, f{'}(z_{0})=\infty^{t_{0}+1}, c(z_{0})f(z_{0})=\infty^{t_{0}}, b(z_{0})f(z_{0})^{2} =\infty^{2t_{0}},\\\nonumber
&a(z_{0})f(z_{0}+1)=b(z_{0})f(z_{0})^{2} =\infty^{2t_{0}}, 
f(z_{0}+1)=\infty^{2t_{0}+\theta_{1}}. 
\end{align}
As previously stated, the definitions of $\theta_{j}$  $(j=1,2)$ can refer to \eqref{N32}.
By shifting \eqref{E10} up, \begin{align*}
{f}'(z+1)=a(z+1)f(z+2)+b(z+1)f(z+1)^{2}+c(z+1)f(z+1)+d(z+1).    
\end{align*}
Then by \eqref{N48} we have 
\begin{align*}
&f{'}(z_{0}+1)=\infty^{2t_{0}+\theta_{1}+1}, b(z_{0}+1)f(z_{0}+1)^{2}=\infty^{4t_{0}+2\theta_{1}}, \\\nonumber
&c(z_{0}+1)f(z_{0}+1)=\infty^{2t_{0}+\theta_{1}},\\\nonumber   
&a(z_{0}+1)f(z_{0}+2)=b(z_{0}+1)f(z_{0}+1)^{2}=\infty^{4t_{0}+2\theta_{1}},f(z_{0}+2)=\infty^{4t_{0}+2\theta_{1}+\theta_{2}}.  
\end{align*}
 By shifting \eqref{E10} up more steps, we can finally have $f(z)$ is not of subnormal growth. Therefore, the generic pole of $f(z)$ must exist.
Since the proof is very similar to that of \textbf{1.2}, we omit the details here.

\textbf{Subcase 2.2:} All coefficients of $R(z,f(z))$ are not all rational functions. Suppose, on the contrary, that $f(z)$ has no generic poles. Let $f(z_{0})=\infty^{t_{0}},$ where $z=z_0$ is any pole of $f(z),$    $t_0$ is a positive integer and $t_{0}>1.$ For the sake of convenience in  calculation, we assume that  $k(z_{0})=0^{1},$ where $k(z)$ is arbitrary coefficient of \eqref{E10}. The proof for the general case follows the same idea. Then from \eqref{E10} and its shifting, we have \begin{align}\label{N51}
&f(z_{0})=\infty^{t_{0}}, c(z_{0})f(z_{0})=\infty^{t_{0}-1}, f'(z_{0})=\infty^{t_{0}+1}, b(z_0)f(z_{0})^{2}=\infty^{2t_{0}-1},\\\nonumber  
 &d(z_{0})=0^{1}, a(z_{0})f(z_{0}+1)=b(z_0)f(z_{0})^{2}=\infty^{2t_{0}-1}, f(z_{0}+1)=\infty^{2t_{0}-2}. 
\end{align}
Shifting \eqref{E10} one step yields
\begin{align*}
{f}'(z+1)=a(z+1)f(z+2)+b(z+1)f(z+1)^{2}+c(z+1)f(z+1)+d(z+1),   
\end{align*}
By \eqref{N51} we have \begin{align}\label{N52}
&f{'}(z_{0}+1)=\infty ^{2t_{0}-1}, b(z_{0}+1)f(z_{0}+1)^{2}=\infty^{4t_{0}-5}, c(z_{0}+1)f(z_{0}+1)=\infty^{2t_{0}-3},\\\nonumber 
&d(z_{0}+1)=0 ^{1}, a(z_{0}+1)f(z_{0}+2)=b(z_{0}+1)f(z_{0}+1)^{2}=\infty^{4t_{0}-5},\\\nonumber
&f(z_{0}+2)=\infty^{4t_{0}-6}.
\end{align}
Shifting \eqref{E10} one step again, we have 
\begin{align*}
{f}'(z+2)=a(z+2)f(z+3)+b(z+2)f(z+2)^{2}+c(z+2)f(z+2)+d(z+2).   
\end{align*}
By \eqref{N52} we obtain that \begin{align}\label{N54}
&f{'}(z_{0}+2)=\infty ^{4t_{0}-5}, b(z_{0}+2)f(z_{0}+2)^{2}=\infty^{8t_{0}-13}, c(z_{0}+2)f(z_{0}+2)=\infty^{4t_{0}-7},\\\nonumber 
&d(z_{0}+2)=0 ^{1}, a(z_{0}+2)f(z_{0}+3)=b(z_{0}+2)f(z_{0}+2)^{2}=\infty^{8t_{0}-13},\\\nonumber
&f(z_{0}+3)=\infty^{8t_{0}-14}.
\end{align}
By \eqref{N51}, \eqref{N52} and \eqref{N54} we note that
\begin{align*}
&a(z_{0})f(z_{0}+1)=\infty^{2t_{0}-1}, \\ &a(z_{0}+1)f(z_{0}+2)=\infty^{4t_{0}-5},\\ &a(z_{0}+2)f(z_{0}+3)=\infty^{8t_{0}-13}. 
\end{align*}
Shifting \eqref{E10} upward by more steps, we obtain that $a(z_{0}+n-1)f(z_{0}+n)=\infty^{2^{n}t_{0}-(2^{n+1}-3)}=\infty^{2^{n}(t_{0}-2)+3},$ where $n$ is a positive integer and $n\ge1.$
Note also that $t_{0}> 2.$ Otherwise, we can conclude that  generic poles of $f(z)$ exist. Combining the fact that $a(z)$ is a small function of $f(z)$ with a  
similar previous proof, we can deduce that $f(z)$ does not exhibit subnormal growth. Therefore, the generic poles of $f(z)$ must exist.

Next, we prove that the generic poles of \(f(z)\) exist in \eqref{E61} and all generic poles are simple. Since the process is very similar to the proof of \eqref{E10}, we will omit some details.
Let $a(z)$ be a non-rational meromorphic function, and let not all coefficients of $R(z,f(z))$ be rational functions. 
Suppose, on the contrary, that \(f(z)\) does not have any generic poles. Since \eqref{N20} holds, and all coefficients of \eqref{E61} are small functions of \(f(z)\), it follows that the order of poles of $f(z)$ at $z=z_0$ is much greater than the order of the zeros or poles of $k(z)$ at $z=z_0,$ where $k(z)$ is arbitrary coefficient of \eqref{E61} and $z=z_0$ is arbitrary pole of $f(z).$ In other words, by \eqref{E61} 
we have following: if \(f(z_{0})=\infty^{l_{0}}\), it follows that
\begin{align}\label{N35}
&{f}'(z_0)=\infty^{l_{0}+1}, \frac{a_{0}(z_0)f(z_0)^{3}+a_{1}(z_0)f(z_0)^{2}+a_{2}(z_0)f(z_0)+a_{3}(z_0)}{f(z_{0})-a_{4}(z_{0})}=\infty^{2l_{0}+\delta_{1}}\\\nonumber    
&a(z_{0})f(z_{0}+1)=\infty^{2l_{0}+\delta_{1}}, f(z_{0}+1)=\infty^{2l_{0}+\delta_{1}+\delta_{2}}.           \\\nonumber
\end{align} 
By shifting \eqref{E61} up, we have following:
\begin{align}\label{N36}
  &f'(z+1) = a(z+1) f(z+2)\\\nonumber
+&\frac{a_0(z+1) f(z+1)^3+ a_1(z+1) f(z+1)^2 + a_2(z+1) f(z+1) + a_3(z+1)}{f(z+1) - a_4(z+1)}.
\end{align}
By \eqref{N35} we have  
\begin{align}\label{N37} 
&f{'}(z_0+1)=\infty^{2l_0+\delta_{1}+\delta_{2}+1},            \\\nonumber
&\frac{a_{0}(z_0+1)f(z_0+1)^{3}+a_{1}(z_0+1)f(z_0+1)^{2}+a_{2}(z_0+1)f(z_0+1)+a_{3}(z_0+1)}{f(z_{0}+1)-a_{4}(z_{0}+1)}\\\nonumber
&=\infty^{4l_{0}+2\delta_{1}+2\delta_{2}+\delta_{3}},\\\nonumber 
&a(z_{0}+1)f(z_{0}+2)=\infty^{4l_{0}+2\delta_{1}+2\delta_{2}+\delta_{3}},\\\nonumber
&f(z_{0}+2)=\infty^{4l_{0}+2\delta_{1}+2\delta_{2}+\delta_{3}+\delta_{4}},
\end{align}  
 where for the definitions of 
 $\delta_{j}\ne 0$ $(j=1,2,3,4)$ the reader can refer to \eqref{N32}.
Like in the previous case, by repeating the above iteration, we find that the multiplicities of the poles of \(f(z+n)\) increase  according to an exponential function of $n$. This leads to a contradiction, as \(f(z)\) is of subnormal growth.
Therefore the generic poles of \eqref{E61} also exist and all generic poles are simple. Then
\begin{align}\label{E835}
N_{1}(r,f(z))=T(r,f(z))+S(r,f(z)) 
\end{align}
is valid for \eqref{E10} and \eqref{E61}.

(ii). 
By making a transformation $w(z)=f(z)b(z),$ and substituting it into \eqref{E10}, it follows that \eqref{E10} takes the form 
\begin{align}\label{E380}
{w}'(z)=A(z)w(z+1)+w(z)^{2}+C(z)w(z)+D(z), 
\end{align}
where \begin{align}\label{E395}
\begin{cases}
A(z)=\frac{a(z)b(z)}{b(z+1)}\not\equiv 0,\\
C(z)=\frac{{b}'(z)}{b(z)}+c(z),\\
D(z)=b(z)d(z).
\end{cases}
\end{align}

By the proof of (i), we find that the generic poles of $f(z)$ must exist and all generic poles are simple in \eqref{E10}. That is \begin{align}\label{N60}
 N_{1}(r,w(z))=T(r,w(z))+S(r,w(z)). \end{align}
Note that since \eqref{E10} is equivalent to \eqref{E380}, it follows that $w(z)$ has generic poles and almost all generic poles of $w(z)$ are simple in \eqref{E380} and there are at most $S(r,w)$ points ${z_{j}}$ such that $w(z_{j})$ are not simple poles.
Let $z=z_{0}$ be arbitrary simple generic pole of $w(z).$ The Laurent series for  $w(z),$ ${w}'(z)$ and $w(z)^{2}$, in a neighborhood of $z_0$, are 
\begin{align}
\label{E28}
w(z)=&\frac{\lambda (z_{0})}{z-z_{0}} +k(z_{0})+O(z-z_{0}),\\\nonumber {w}'(z)=&\frac{-\lambda (z_{0})}{(z-z_{0})^{2}}+O(z-z_{0}),\\\nonumber
 w(z)^{2}=&\frac{\lambda(z_{0})^{2}}{(z-z_{0})^{2}}+\frac{2k(z_{0})\lambda (z_{0})}{z-z_{0}}+O(1), 
\end{align}
 where 
 $\lambda (z_{0})$ depends on $z_{0}.$ Since we  only need to consider the simple poles of $w(z),$ then we can always assume that $\lambda(z_{0})\neq 0.$  
Substituting \eqref{E28}
into \eqref{E380} results in $-\lambda (z_{0})-\lambda (z_{0})^{2}=0,$ namely,
 $\lambda (z_{0})=-1.$ If $-\lambda (z_{0})-\lambda (z_{0})^{2}\neq 0,$  then by \eqref{E380} it follows that $w(z+1)$ has a double pole at $z=z_{0}.$ 
By shifting \eqref{E380} up, we obtain the following:
\begin{align}\label{N40}
A(z+1)w(z+2) & = {w}' (z+1)-w(z+1)^{2}-C(z+1)w(z+1)-D(z+1).
\end{align}
We use $k(z+1)$ to represent an arbitrary  coefficient of \eqref{N40}. 
Supposing that $k(z_{0}+1)=0^{l}$ or $k(z_{0}+1)=\infty^{l},$  we define 
\begin{align*}
K=\left\{z_{0}:w(z_{0})=\infty^{1};k(z_{0}+1)=0^{l},or,k(z_{0}+1)=\infty^{l}, l\geq 1\right\}.
\end{align*}
Combining \eqref{N60} with the fact that all coefficients of \eqref{N40} are small functions of $w(z),$ we obtain that the integrated counting function for the set $K,$ denoted by $N_{K}(r,w(z)),$ is at most $S(r,w(z)).$ 
Therefore, by \eqref{N60} it follows that there exists a set  \begin{align*}
\hat {K}=\left\{z_{0}:w(z_{0})=\infty^{1}; k(z_{0}+1)\neq 0,k(z_{0}+1)\neq \infty \right\}
\end{align*}
such that the integrated counting function for set ${\hat {K}}$ satisfies \begin{align}\label{N61}
 N_{{\hat {K}}}(r,w(z))=T(r,w(z))+S(r,w(z)). \end{align}
By combining \eqref{N40} with \eqref{N61} we have $z=z_{0}+2$ is also a pole of $w(z)$ and its order is $4.$ By shifting \eqref{N40} up we have\begin{align}\label{N41}
A(z+2)w(z+3)& = {w}' (z+2)-w(z+2)^{2}-C(z+2)w(z+2)-D(z+2).
\end{align}
We use $k(z+2)$ to denote an arbitrary coefficient of \eqref{N41}. Supposing that $k(z_{0}+2)=0^{t}$ or $k(z_{0}+2)=\infty^{t},$ we define \begin{align*}
P=\left\{z_{0}:w(z_{0})=\infty^{1};k(z_{0}+2)=0^{l},or,k(z_{0}+2)=\infty^{l}, l\geq 1\right\},
\end{align*}
and 
\begin{align*}
{\hat P}=\left\{z_{0}:w(z_{0})=\infty^{1};k(z_{0}+2)\neq 0,k(z_{0}+2)\neq \infty^{l}\right\}.
\end{align*}
By the analysis above we find that the integrated counting function for the set $P,$ denoted by $N_{P}(r,w),$ is at most $S(r,w(z)).$ Additionally, the integrated counting function for the set ${\hat P},$ denoted by $N_{\hat P}(r,w),$ satisfies
\begin{align*}
N_{\hat P}(r,w(z))=T(r,w(z))+S(r,w(z)).   \end{align*}
Thus, $z=z_{0}+3$ is also a pole of $w(z)$ and $w(z_{0}+3)=\infty^{8}.$ By continuing to shift \eqref{N41} up , we can eventually obtain that $z=z_0,$ is also a pole of $w(z+n)$ and $w(z_{0}+n)=\infty ^{2^{n}},$ where $n$ is an arbitrary positive integer.
   In this case,
it is easy to verify that  $w(z)$ is not of subnormal growth. Therefore, $\lambda(z_{0})=-1$ is proved.

On the other hand,
 shifting \eqref{E380} down, we have \begin{align}\label{E320}
 A(z-1)w(z)=w'(z-1)-w(z-1)^{2}-C(z-1)w(z-1)-D(z-1).  
\end{align}
By the analysis above we have that $z=z_{0}$ is also a pole of $w(z-1)$. Assume that $w(z_{0}-1)=\infty^{t}$ $(t\ge 1).$ 
By \eqref{E320} we have  if $t> 1,$ then the order of right-hand side of \eqref{E320} is $2t$ at $z=z_{0},$ which yields a contradiction with the order of left-hand side of \eqref{E320} is simple at $z=z_{0}.$  Therefore,
all simple generic poles of $w(z)$ must be simple  poles of $w(z-1).$
 Then we can write $w(z-1)$ in the form \begin{align}\label{E325}
w(z-1)=\frac{\beta (z_{0})}{z-z_{0}}+O(1), \quad \beta (z_{0})\ne 0,
\end{align}
in a small neighborhood of $z=z_{0},$ where $z_0$
is an arbitrary generic pole of $w(z).$ Substituting \eqref{E325} into \eqref{E320}, we obtain $\beta (z_{0})=-1.$ 
By shifting \eqref{E320} down, we have \begin{align}\label{N42}
 A(z-2)w(z-1)=w'(z-2)-w(z-2)^{2}-C(z-2)w(z-2)-D(z-2).  
\end{align} 
By the same way, we have that $z=z_{0}-2$ is also a pole of $w(z).$ Suppose that \begin{align}\label{N63}
 w(z-2)=\frac{\overline \beta }{z-z_{0}}+O(1),  
\end{align}
in a small neighborhood of $z=z_{0}.$ By substituting \eqref{N63} into \eqref{N42}, we have $\overline \beta=-1.$
By continuing to shift \eqref{N42} down, we can end up with \begin{align}\label{E362}
 w(z-n)=\frac{-1}{z-z_{0}}+O(1),
\end{align}
 in a neighborhood of $z_{0},$ where 
$n$ is an arbitrary non-negative integer. 
Next we will prove that \eqref{E362}  also holds for arbitrary negative integer $n$, i.e., we will next prove that \eqref{E362} is valid for all  integers $n$. Substituting \eqref{E28}
into \eqref{E380} we obtain
\begin{align}\label{E363} 
w(z+1)=&\frac{1}{A(z)}\left[{w}'(z)-w(z)^{2}-C(z)w(z)-D(z)\right]\\\nonumber
=&\frac{1}{A(z_{0})}\frac{2k(z_{0})+C(z_{0})}{z-z_{0}}+O(1)\\\nonumber
=&\frac{l(z_{0})}{z-z_{0}}+O(1), 
\end{align}
where we have used the complex geometric series to deduce that
$$\frac{1}{1+O(z-z_{0})}=1+O(z-z_{0}) $$
in a small enough neighborhood of $z=z_{0}$ and $l(z_{0})=\frac{2k(z_{0})+C(z_0)}{A(z_{0})}.$ 
Shifting \eqref{E380} up, it follows that 
\begin{align}\label{E599}
{w}'(z+1)=A(z+1)w(z+2)+w(z+1)^{2}+C(z+1)w(z+1)+D(z+1).  
\end{align}
By substituting \eqref{E363} into \eqref{E599}, we obtain that  \begin{align}\label{E587}
&w(z+2)\\\nonumber
=&\frac{1}{A(z+1)}\left [ {w}'(z+1)-w(z+1)^{2}-C(z+1)w(z+1)-D(z+1) \right ]\\\nonumber
=&\frac{1}{A(z_0+1)}\left [ \frac{-l(z_{0})}{(z-z_0)^2}-\frac{l^2(z_0)}{(z-z_0)^2}-C(z_0+1)\frac{l(z_0)}{(z-z_0)}-D(z_0+1)+O(1) \right ]
\end{align} in a small neighborhood of $z_0.$ By \eqref{E587} we have \(-l(z_0)-l(z_0)^{2}=0\). Otherwise, we have \(w(z+2)\) has a pole of order \(2\) at \(z=z_0\). By shifting \eqref{E587} one step, we have
\begin{align}\label{N21}
&w(z+3)\\\nonumber
=&\frac{1}{A(z+2)}\left [ {w}'(z+2)-w(z+2)^{2}-C(z+2)w(z+2)-D(z+2)\right]\\\nonumber
\end{align}
in a neighborhood of $z_{0}.$
From \eqref{N21}, we have \(w(z+3)\) has a pole of order \(2^2\) at \(z=z_0\). Repeating this process, \(w(z+n)\) has a pole of order \(2^{n-1}\) at \(z=z_0\), and then it follows that \(w(z)\) is not of subnormal growth. Therefore, $l(z_0)=0$
or $l(z_0)=-1.$  Hence \begin{equation}\label{N22}
w(z+2)=\frac{-C(z_0+1)}{A(z_0+1)}\frac{l(z_0)}{(z-z_0)}+O(1):=\frac{m(z_0)}{z-z_0}+O(1)\end{equation}   
in a neighborhood of $z_{0}.$
Substituting \eqref{N22} into \eqref{N21}, it follows that 
\begin{align*}
&w(z+3)\\\nonumber
=&\frac{1}{A(z+2)}\left [ {w}'(z+2)-w(z+2)^{2}-C(z+2)w(z+2)-D(z+2) \right ]\\
\nonumber
=&\frac{1}{A(z_0+2)}\left [ \frac{-m(z_{0})}{(z-z_0)^2}-\frac{m^2(z_0)}{(z-z_0)^2}-C(z_0+2)\frac{m(z_0)}{(z-z_0)}-D(z_0+2)+O(1) \right ]
\end{align*}
in a neighborhood of $z=z_{0}.$
Similarly, we have $m(z_0)=-1$ or $m(z_0)=0.$ Then we have 
\begin{equation} w(z+3) =\frac{\alpha (z_0)}{z-z_0}+O(1)\end{equation} in a neighborhood of $z=z_{0},$  where $\alpha(z_0)=\frac{-C(z_0+2)m(z_0)}{A(z_0+2)}.$
By repeating  the process above, we have that 
\begin{align}\label{E593}
w(z+\hat {k})=\frac{\eta (z_0)}{z-z_0}+O(1),   
\end{align}
in a small neighborhood of $z_0,$ where $\eta (z_0)=0$ or $-1,$ and $\hat {k}$ is an arbitrary positive 
integer. In \eqref{E593},
suppose that for an arbitrary positive integer $\hat {k},$ we have $\eta(z_0)=-1$. Then 
\begin{align}\label{E581} 
w(z+\hat {k})=\frac{-1}{z-z_0}+O(1),
\end{align}
in a neighborhood of $z_0.$
Combining  \eqref{E362} and \eqref{E581}, we  obtain that \eqref{E362} is valid for all integers $n.$ 
Then
\begin{align}\label{E493}
 N(r,w(z+1)-w(z))=S(r,w(z)).   
\end{align}

On the other hand, in \eqref{E593}, if there exists a positive integer $k$ such that $\eta (z_0)=0,$ i.e., $w(z_0+k)\not =\infty.$ Then by shifting \eqref{E380} up, we obtain $w(z_0+k+n)\not=\infty,$ where $n$ is an arbitrary positive integer. Without loss of generality,
we assume now that $k=1,$ i.e., $w(z_0)=\infty$ and $w(z_0+n)\not=\infty,$ where $n$ is an arbitrary positive integer. Additionally, combining  \eqref{E362} we have \begin{align}\label{E582}
\dots ,w(z_0-2)=\infty ,w(z_0-1)=\infty ,w(z_0)=\infty, w(z_0+1)\not =\infty, w(z_0+2)\not =\infty,\dots  
\end{align}
By the previous analysis we see that if $z=z_{0}$ is a generic pole of $w(z),$ then $z=z_{0}-1,$ $z=z_{0}-2,$ \dots, are all generic poles of $w(z).$
Moreover, we call $z=z_{0}$ an $\mathit{end}$ $\mathit{point}$ of $w(z)$, if $z=z_{0}$ is a generic pole of $w(z)$ and $w(z_{0}+1)\not=\infty.$ 
Then by \eqref{E582} we see that $z=z_0$ is an end point of $w(z).$  
Additionally, $N_{e}(r,w)$  stands for the counting function for the end points of $w(z)$ and $N_{p}(r,w)$ stands for the counting function for the regular poles of $w(z).$ (We say that $z=z_{0}$ is a $\mathit{regular}$ $\mathit{pole}$ of $w(z)$ if $z=z_{0}$ is a generic pole of $w(z)$ that is not an end point of $w(z)$.) In other words, the generic poles of $w(z)$ are a  combination of the end points and regular poles of $w(z).$
Combining the  definitions of end points and regular poles, \eqref{N60} and \eqref{E582}, yields 
\begin{align}\label{E381}
N(r,w(z))=&N_{1}(r,w(z))+S(r,w(z))\\\nonumber
=&N_{e}(r,w(z))+N_{p}(r,w(z))+S(r,w(z)).
\end{align}
Note that  $z=z_{0}$ is an end point and  $z=z_{0}-n$ is regular pole, where $n$ is a positive integer. Hence, from \eqref{E582} we have following inequality
\begin{align}\label{E382}
n_{e}(r,w(z)) \le \frac{1}{m}n_{p}(r+m,w(z))
\end{align} for any positive integer $m.$
By integration from the both sides of \eqref{E382} and using Lemma \ref{L2} and \eqref{E835}, we have that  
\begin{align}\label{E383}
N_{e}(r,w(z)) \le& \left (\frac{1}{m}+\varepsilon\right)N_{p}(r+m,w(z))\\\nonumber
 \le & \left (\frac{1}{m}+\varepsilon\right)N(r+m,w(z))\\\nonumber
=&\left (\frac{1}{m}+\varepsilon\right)N(r,w(z))+S(r,w(z))\\\nonumber
=&\left (\frac{1}{m}+\varepsilon\right)T(r,w(z))+S(r,w(z)).
\end{align}
Assuming that $N_e(r,w)\not=S(r,w)$, it follows that there exists a set $F$ of infinite logarithmic measure (or of non-zero upper density measure) such that we can find a positive constant $K$ which satisfies \begin{align}\label{E384}
N_{e}(r,w(z))\ge K T(r,w(z)) 
\end{align}
for all $r\in F.$ Combining \eqref{E383} and \eqref{E384}, it follows that 
\begin{align}\label{E385}
KT(r,w(z))\le N_{e}(r,w(z))\le\left ( \frac{1}{m}+\varepsilon\right)T(r,w(z))+S(r,w(z))   
\end{align}
for all $r\in F$ is valid. Since $K$ is a fixed constant, we can choose $m$ large enough and $\varepsilon > 0$ small enough to get a contradiction. Therefore, 
$$N_e(r,w)=S(r,w),$$
which yields \eqref{E362} holds for all integers $n,$  where $z_{0}$ is arbitrary generic pole of $w(z)$. Then \eqref{E493} is always valid.\par

On the other hand, by \eqref{E835} we have $m(r,w(z))=S(r,w(z)).$ Then by Lemma \ref{L3}, it follows  that \begin{align}\label{E584} 
m(r,w(z+1)-w(z))=S(r,w(z)).
\end{align}
Combining \eqref{E493} and \eqref{E584},
we can therefore obtain that  
\begin{align}\label{E369}
w(z+1)-w(z):=\gamma (z),\quad T(r,\gamma (z))=S(r,w(z)).
\end{align}
Noting that $w(z)=f(z)b(z),$ we have \eqref{E369} is equivalent to \begin{align}\label{E400}
f(z+1)=\frac{f(z)b(z)+\gamma (z)}{b(z+1)}.
\end{align}
Substituting \eqref{E400} into \eqref{E10}, it follows that \eqref{E10} takes the form
$${f}'(z)=b(z)f(z)^{2}+\left [ c(z)+\frac{a(z)b(z)}{b(z+1)}\right ]f(z)+d(z)+\frac{a(z)\gamma (z)}{b(z+1)} ,$$
where $n(z)=d(z)+\frac{a(z)\gamma (z)}{b(z+1)},$
and $T(r,n(z))=S(r,f(z)).$

(iii). 
Let $f(z)-a_{4}(z)=g(z).$ Then \eqref{E61} becomes \begin{align}\label{E600}
{g}'(z)=a(z)g(z+1)+a_{0}(z)g(z)^{2}+\hat{a} _{1}(z)g(z)+\hat{a} _{2}(z)+\frac{\hat{a}_{3}(z)}{g(z)},   
\end{align}
where
\begin{align}
\begin{cases}\label{N612}
 \hat a_1(z)=a_1(z)+3a_0(z)a_4(z),\\
 \hat a_2(z)=a_2(z)+2a_1(z)a_4(z)+3a_0(z)a_4(z)^{2}+a(z)a_{4}(z+1)-a'_4(z),\\
\hat a_3(z)=a_0(z)a_4(z)^{3}+a_1(z)a_4(z)^{2}+a_2(z)a_4(z)+a_3(z). 
\end{cases}
\end{align}
Furthermore, assume that $g(z)a_{0}(z)=u(z)$. By a  simple computation  we then see that \eqref{E600} becomes \begin{align}\label{E601} 
{u}'(z) =A(z)u(z+1)+u(z)^{2}+C(z)u(z)+D(z)+\frac{E(z)}{u(z)},  
\end{align}
where \begin{align}\label{E602}
\begin{cases}
 A(z)=\frac{a(z)a_0(z)}{a_0(z+1)}\not\equiv 0,\\
C(z)=\hat a_{1}(z)+\frac{a'_{0}(z)}{a_0 (z)},\\
D(z)=\hat a_2(z)a_0(z),\\
E(z)=\hat a_3(z)a_0(z)^{2}. 
\end{cases}
\end{align}
In what follows, we first show that \eqref{E601} reduces to a Riccati equation. Also we note that the relations between $f(z)$ and $u(z)$ satisfy $u(z)=a_{0}(z)f(z)-a_{4}(z),$ that is to say, \eqref{E61} can also reduces to a Riccati equation.\par

We discuss two cases.\par 
Case 1. Assume that $E(z)\equiv 0$, and thus $\hat{a}_{3}(z)\equiv 0.$ Then \eqref{E601} reduces to  \begin{align}\label{N23} 
{u}'(z) =A(z)u(z+1)+u(z)^{2}+C(z)u(z)+D(z).
\end{align}
Note that \(A(z)=\frac{a(z)a_0(z)}{a_0(z+1)}\not\equiv 0\) and  \eqref{N23} is the same kind of equation as  \eqref{E380}.
Then by the proof of (ii), we get that
\begin{align}\label{N24}
u(z+1)=u(z)+\varphi (z),\end{align} where $T(r,\varphi (z))=S(r,u(z)). $
Substituting \eqref{N24} into \eqref{N23}, we obtain \eqref{N23} reduces into following Riccati equation:
\begin{equation}\label{N25}
{u}'(z)=u(z)^{2}+u(z)(A(z)+C(z))+A(z)\varphi (z)+D(z).    
\end{equation} By $u(z)=(f(z)-a_4(z))a_0(z),$ we continue to get the  Riccati equation:
\begin{align*}
f{'}(z)
=&f(z)^{2}+\left(\frac{a_{0}'(z)}{a_{0}(z)}-2a_{4}(z)+A(z)+C(z)\right)f(z)+a_{4}(z)+\frac{a'_{0}(z)}{a_{0}(z)}a_{4}(z)\\\nonumber   
&+a_{4}(z)^{2}-a_{4}(z)(A(z)+C(z))+\frac{A(z)(\varphi (z)+D(z))}{a_{0}(z)}  \\\nonumber
&=f(z)^{2}+h_{1}(z)f(z)+h_{2}(z),   
\end{align*}
where \begin{align}\label{N611}
&h_1(z) =\frac{a'_0(z)}{a_0(z)}-2a_4(z)+A(z)+C(z),\\\nonumber
&h_2(z) =a_{4}(z)+\frac{a'_{0}(z)}{a_{0}(z)}a_{4}(z)+a_{4}(z)^{2}-a_{4}(z)(A(z)+C(z))+\frac{A(z)(\varphi (z)+D(z))}{a_{0}(z)},\\\nonumber
\end{align}
and $T(r,h_{2}(z))=S(r,f(z)).$

Next, we will proceed to demonstrate that \eqref{N23} reduces to a Riccati equation having different coefficients with Case $1$.

Case 2. Assume now that $E(z)\not\equiv 0$, and thus $\hat{a}_{3}(z)\not\equiv 0.$ Then $u=0$ is not a root of $$P(z, u(z)):=u(z)\left({u}'(z)-A(z)u(z+1)-u(z)^{2}-C(z)u(z)-D(z)\right)-E(z)=0.$$ By applying Lemma \ref {L10} to \eqref{E601}, it follows that 
\begin{align}\label{E331}
 N\left ( r,\frac{1}{u(z)}\right)=T(r,u(z))+S(r,u(z)).
\end{align}
\par
Now we assert that  generic zeros of $u(z)$ must exist when $u(z)$ is a subnormal meromorphic solution of \eqref{E601}. Suppose, on the contrary, that the generic zeros of $u(z)$ do not exist. Let $z=z_{1}$ be a zero of order $m_{1}$ of $u(z).$ Then from \eqref{E601} we have 
$\frac{E(z_1)}{u(z_1)}=\infty^{m+\varepsilon},$ where $\left | \varepsilon  \right |\ll m ,$ i.e.,  for each fixed $\epsilon\in (0,1),$ there exists $z_{1} $ such that $\left|\varepsilon\right|\leq \epsilon \cdot m.$ Note that $a(z)$ and $a_0(z)$ are rational, and so $A(z)$ must be rational by \eqref{E602}. Hence, by $\frac{E(z_1)}{u(z_1)}=\infty^{m+\varepsilon}$ we deduce that $A(z_1)u(z_1+1)=u(z_1+1)=\infty^{m+\varepsilon}$ in \eqref{E601}. Shifting \eqref{E601} up, we have
\begin{align}\label{E744}
{u}'(z+1)
=&A(z+1)u(z+2)+u(z+1)^{2}+C(z+1)u(z+1)\\\nonumber
&+D(z+1)+\frac{E(z+1)}{u(z+1)}.  
\end{align}
Then $u'(z_1+1)=\infty^{m+\varepsilon+1},$ and $A(z_1+1)u(z_1+2)=u(z_1+2)=\infty^{2m+2\varepsilon}.$ By continuing to shift \eqref{E601}, we have
\begin{align*}
{u}'(z+2)
=&A(z+2)u(z+3)+u(z+2)^{2}+C(z+2)u(z+2)\\\nonumber
&+D(z+2)+\frac{E(z+2)}{u(z+2)}.  
\end{align*}
Then $u'(z_1+2)=\infty^{2m+2\varepsilon+1},$ and $A(z_1+2)u(z_1+3)=u(z_1+3)=\infty^{4m+4\varepsilon}.$ By repeating this process, we have that $A(z_1+d)u(z_1+d+1)=u(z_1+d+1)=\infty^{2^{d}m+2^{d}\varepsilon}$ for arbitrary positive integer $d.$
It is easy to verify that this contradicts the assumption that $u(z)$ is subnormal. Thus, the existence of generic zero of $u(z)$ is proved in \eqref{E601}.

 We assume now that \begin{align}\label{E332}
u(z)=k(z_{0})(z-z_{0})^{m}+O((z-z_{0})^{m+1}), \quad k(z_{0})\neq 0,
\end{align}
in a small neighborhood of $z_0,$ where
$z_0$ is an arbitrary generic zero of $u(z)$.
If $m>1,$ then by shifting \eqref{E601} forward  more steps, we can derive that the growth of $u(z)$ is not subnormal  (similarly to \eqref{E744}). This implies $m=1,$ i.e.,
\begin{align} 
\label{E337} 
N_{1}\left(r,\frac{1}{u(z)}\right)=T(r,u(z))+S(r,u(z)).
\end{align}
Substituting \eqref{E332} with $m=1$ into \eqref{E601}, we have that
\begin{align}\label{E334}
u(z+1)=&\frac{1}{A(z)}\left[{u}'(z)-u(z)^{2}-C(z)u(z)-D(z)-\frac{E(z)}{u(z)}\right]\\\nonumber
=&\frac{1}{A(z_{0})}\frac{-E(z_{0})}{k(z_{0})}\frac{1}{z-z_{0}}+\eta (z_0)+O(z-z_0)\\\nonumber
=&\frac{l(z_0)}{z-z_0}+\eta (z_0)+O(z-z_0),    
\end{align}
in a small neighborhood of $z_0,$ where $l(z_{0})=\frac{1}{A(z_0)}\frac{-E(z_0)}{k(z_0)}.$
Shifting \eqref{E334} forward one step
\begin{align}\label{E335}
&u(z+2)\\\nonumber
=&\frac{1}{A(z+1)}u'(z+1)-\frac{1}{A(z+1)}u(z+1)^{2}-\frac{1}{A(z+1)}C(z+1)u(z+1)\\\nonumber
&-\frac{1}{A(z+1)}D(z+1)-\frac{1}{A(z+1)}\frac{E(z+1)}{u(z+1)} 
\end{align}
Substituting \eqref{E334} into \eqref {E335}
we have 
$l(z_0)=-1.$
Otherwise, $u(z+2)$ has a pole of order $2$ at $z=z_0.$ Then
by shifting \eqref{E601} forward  more steps, we have  
\begin{align*}
&u(z+3)\\
=&\frac{1}{A(z+2)}u'(z+2)-\frac{1}{A(z+2)}u(z+2)^{2}-\frac{1}{A(z+2)}C(z+2)u(z+2)\\\nonumber
&-\frac{1}{A(z+2)}D(z+2)-\frac{1}{A(z+2)}\frac{E(z+2)}{u(z+2)},\\\nonumber  
\end{align*}
in this case, it follows that $u(z+3)$ has a pole of order $4$ at $z=z_0.$ If we repeat the process,  we can then see that $u(z+n)=\infty^{2^{n-1}}$ at $z=z_{0} .$ By performing a similar computation to   \eqref{N46}, we can obtain that the growth of $u(z)$ exceeds subnormal growth. Substituting $l(z_0)=-1$
into \eqref{E334} and by simple calculation, we have   
\begin{align}\label{E830}
{u}'(z+1)=\frac{-l(z_0)}{(z-z_0)^{2}}+O(1)=\frac{1}{(z-z_0)^{2}}+O(1),
\end{align}
and \begin{align}\label{E831}
u(z+1)^{2}=&\frac{l(z_0)^{2}}{(z-z_0)^{2}}+\frac{2l(z_0)\eta (z_0)}{z-z_0}+O(1)\\\nonumber
=& \frac{1}{(z-z_0)^{2}}-\frac{2\eta (z_0)}{z-z_0}+O(1),    
\end{align}
in a small neighborhood of $z_0.$
With \eqref{E335}, \eqref{E830} and \eqref{E831}, it follows that \begin{align}\label{E832} 
u(z+2)=\frac{2\eta (z_0)+C(z_0+1)}{A(z+1)}\frac{1}{z-z_0}+O(1),
\end{align}
in a small neighborhood of $z_0.$
Shifting \eqref{E335} forward one step to obtain the expression of $u(z+3)$ and substituting \eqref{E832} into the expression of $u(z+3),$ implies $$\frac{2\eta (z_0)+C(z_0+1)}{A(z+1)}+\left (\frac{2\eta (z_0)+C(z_0+1)}{A(z+1)}  \right )^{2}=0  .$$ Then
\eqref{E832} becomes \begin{align}\label{N1} 
u(z+2) & = \frac{-1}{z-z_0}+O(1),
\end{align} or 
\begin{align}\label{N2} 
u(z+2) & = O(1),
\end{align}
in a small neighborhood of $z_0.$

Next, we present all cases concerning the distribution of zeros and poles for $u(z).$ According to \eqref{E835} and \eqref{E337}, it is sufficient to consider only the simple zeros and poles of $u(z).$ Let $z=z_{0}$ be an arbitrary simple generic zero of $u(z).$ By combining \eqref{N1} and  \eqref{N2}, we find three possibilities for $u(z_{0}+2):$ $u(z_{0}+2)=0^{1},$ $u(z_{0}+2)=\infty^{1},$ and $u(z_{0}+2)$ is finite and non-zero. Furthermore, we will demonstrate that at most $S(r,u)$ points satisfy $u(z_{0}+2)=\infty^{1}.$

\textbf{Step 1.} 
Assume that $u(z_{0}+2)=\infty^{1}.$ Then the  distribution of the zeros and poles of $u(z)$ is as follows: 

\textbf{Case 1.1.}
\begin{align*}
\dots u(z_0-1)=\Delta_1 ,u(z_0)=0^{1},u(z_0+1)=\infty^{1} , u(z_0+2)=\infty^{1},u(z_0+3)=\infty ^{1} \dots   
\end{align*}

 \textbf{Case 1.2.} 
 \begin{align*}
\dots u(z_0-1)=\infty^{1},u(z_0)=0^{1},u(z_0+1)=\infty^{1} , u(z_0+2)=\infty^{1},u(z_0+3)=\Delta_2 \dots   \end{align*}

\textbf{Case 1.3.}
\begin{align*}
\dots u(z_0-1)=\infty^{1},u(z_0)=0^{1},u(z_0+1)=\infty^{1} , u(z_0+2)=\infty^{1},u(z_0+3)=\infty ^{1} \dots   
\end{align*}

\textbf{Case 1.4.} 
\begin{align*}
\dots  u(z_0-1)=\Delta_1,u(z_0)=0^{1},u(z_0+1)=\infty^{1} , u(z_0+2)=\infty^{1},u(z_0+3)=\Delta_2 \dots    
\end{align*}
where $\Delta_1$
denotes a finite and non-zero value, and $\Delta_2$ is a finite value.
It is easy to verify that 
 \begin{align}\label{E837} 
N_{\mathbf{Step 1} } \left ( r,\frac{1}{u}\right)
=&N_{\textbf{Case 1.1}}\left ( r,\frac{1}{u} \right )+ N_{\textbf{Case 1.2}}\left ( r,\frac{1}{u} \right )\\\nonumber
\quad&+N_{\textbf{Case 1.3}}\left ( r,\frac{1}{u} \right )+N_{\textbf{Case 1.4}}\left (r,\frac{1}{u} \right ).    
\end{align}

\textbf{Case A.}
Let the distribution of the zeros and  poles of $u(z)$ satisfy case $1.$
In other words, we can always find that there exist at least three simple poles that can be paired with each simple zero of $u(z).$ Then combining \eqref{E337} and Lemma \ref{L2}, it follows that   
\begin{align*}
N_{\textbf{case 1.1}}\left ( r,\frac{1}{u}\right)=&N_{\mathbf{Step 1} } \left ( r,\frac{1}{u}\right)\\\nonumber
\le& \frac{1}{3}N_{\mathbf{Step 1}}(r+3,u)+S(r,u).\\\nonumber
\end{align*}
\textbf{Case B.} Let the distribution of the zeros and poles of $u(z)$ satisfy case $2.$
Then we can see that there are at least two simple poles for each simple zero of $u(z).$ Then, by combining this fact and\eqref{E337} with Lemma \ref{L2}, we arrive at the following expression:
\begin{align*}
N_{\textbf{case 1.2}}\left ( r,\frac{1}{u}\right)=&N_{\mathbf{Step 1} }\left ( r,\frac{1}{u}\right)\\\nonumber 
\le& \frac{1}{2}N_{\mathbf{Step 1} }(r+2,u)+S(r,u).\\\nonumber 
\end{align*}

\textbf{Case C.} Let the distribution of the zeros and poles of $u(z)$ satisfy case $3.$ By above discussion from case $A$ and case $B$. It is easy to verify that \begin{align*}
N_{\textbf{case 1.3}}\left ( r,\frac{1}{u}\right)=&N_{\mathbf{Step 1}}\left ( r,\frac{1}{u}\right)\\\nonumber
\le& \frac{1}{3}N_{\mathbf{Step 1} }(r+3,u)+S(r,u).     
\end{align*}

\textbf{Case D.} Let the distribution of the zeros and poles of $u(z)$ satisfy case $4.$ Likewise, we have that \begin{align*}
 N_{\textbf{case 1.4}}\left ( r,\frac{1}{u}\right)=&N_{\mathbf{Step1} } \left ( r,\frac{1}{u}\right)\\\nonumber
\le & \frac{1}{2}N_{\mathbf{Step1} }(r+2,u)+S(r,u).
\end{align*}

\textbf{Case E.} Let the distribution of the zeros and poles of $u(z)$ be an arbitrary  combination of cases A through D. Then Combining \eqref{E837} with the conclusions of case $A$ through case $D$ result in 
\begin{align}\label{EN3} 
N_{\mathbf{Step 1} } \left (r,\frac{1}{u}\right)
\le\max  \left \{ \frac{1}{3},\frac{1}{2},\frac{1}{3},\frac{1}{2}\right \} N_{\mathbf{Step1} } (r+3,u)+S(r, u).  
\end{align}

$\textbf{Step 2.}$ 
Let $u(z_0+2)$ be finite and non-zero. Then the distribution of the zeros and poles of $u(z)$ is as follows:

\textbf{Case 2.1.}
$$\dots u(z_0-1)=\infty^{1}, u(z_0)=0^{1}, u(z_0+1)=\infty^{1}, u(z_0+2)=\Delta, u(z_0+3)\not=\infty\dots $$

\textbf{Case 2.2.}
$$\dots u(z_0-1)\not =\infty , u(z_0)=0^{1}, u(z_0+1)=\infty^{1}, u(z_0+2)=\Delta, u(z_0+3)=\infty\dots $$

\textbf{Case 2.3.}
$$\dots u(z_0-1)=\infty^{1} , u(z_0)=0^{1}, u(z_0+1)=\infty^{1}, u(z_0+2)=\Delta, u(z_0+3)=\infty\dots $$

\textbf{Case 2.4.}
$$\dots u(z_0-1)\not =\infty , u(z_0)=0^{1}, u(z_0+1)=\infty^{1}, u(z_0+2)=\Delta, u(z_0+3)\not =\infty\dots $$
where $\Delta$ denotes a finite and non-zero value. 
We now define two sets
\begin{align}\label{E853}
&A_{1}: =\left \{ z_{0}:u(z_0)=0, u(z_0+1)=\infty, u(z_0+2)=\Delta, u(z_0+3)=\infty  \right \},\\\nonumber  
&A_{2}: =\left \{ z_{0}:u(z_0)=0, u(z_0+1)=\infty, u(z_0+2)=\Delta, u(z_0+3)\not =\infty  \right \}.\\\nonumber\end{align}
Next, we define two sets, $\hat{A}_1$ and $\hat{A}_2,$ which are formed out of the poles associated with the zeros of $A_1$ and $A_2$:
\begin{align}\label{E854}
&\hat{A}_1: =
\left \{ z_0+1:u(z_0)=0,u(z_0+1)=\infty, u(z_0+2)=\Delta, u(z_0+3)=\infty  \right \}\\\nonumber 
&\cup\left \{ z_0+3:u(z_0)=0,u(z_0+1)=\infty, u(z_0+2)=\Delta, u(z_0+3)=\infty  \right \},\\\nonumber
&\hat{A}_2: =
\left \{ z_0+1:u(z_0)=0,u(z_0+1)=\infty, u(z_0+2)=\Delta, u(z_0+3)\not =\infty  \right \},\\\nonumber 
\end{align} By the definitions of \eqref{E853} and \eqref{E854}, we have
\begin{align}\label{E160}
&n_{\textbf{Case 2.2.}}\left(r,\frac{1}{u}\right)\le \frac{1}{2}n_{\textbf{Case 2.2.}}(r+3,u),\\\nonumber
&n_{\textbf{Case 2.3.}}\left(r,\frac{1}{u}\right)\le \frac{1}{2}n_{\textbf{Case 2.3.}}(r+3,u),\\\nonumber
&n_{A_1}\left(r,\frac{1}{u}\right)=n_{\textbf{Case 2.2.}}\left(r,\frac{1}{u}\right)+n_{\textbf{Case 2.3.}}\left(r,\frac{1}{u}\right),\\\nonumber  
&n_{\hat{A}_1}\left(r+3,u\right)=n_{\textbf{Case 2.2.}}\left(r+3,u\right)+n_{\textbf{Case 2.3.}}\left(r+3,u\right),
\end{align}
and 
\begin{align}\label{E161}
&n_{\textbf{Case 2.1}}\left ( r,\frac{1}{u}\right)\le n_{\textbf{Case 2.1}}\left ( r+1,u\right),\\\nonumber  
&n_{\textbf{Case 2.4}}\left ( r,\frac{1}{u}\right)\le n_{\textbf{Case 2.4}}\left ( r+1,u\right),\\\nonumber
&n_{A_2}\left(r,\frac{1}{u}\right)=n_{\textbf{Case 2.1}}\left ( r,\frac{1}{u}\right)+n_{\textbf{Case 2.4}}\left (r,\frac{1}{u}\right),\\\nonumber
&n_{\hat A_2}\left(r+1,u\right)=n_{\textbf{Case 2.1}}\left ( r+1,u\right)+n_{\textbf{Case 2.4}}\left (r+1,u\right).
\end{align}
Therefore, by \eqref{E160} and \eqref{E161} we have the following:
\begin{align}\label{E855}
n_{A_{1}}\left(r,\frac{1}{u}\right)=&n_{\textbf{Case2.2}}\left(r,\frac{1}{u}\right)+n_{\textbf{Case2.3}}\left(r,\frac{1}{u}\right)\\\nonumber   
\leq&\frac{1}{2}n_{\textbf{Case2.2}}(r+3,u)+\frac{1}{2}n_{\textbf{Case2.3}}(r+3,u)\\\nonumber   
=&\frac{1}{2}n_{\hat A_{1}}(r+3,u), \\\nonumber
n_{A_{2}}\left(r,\frac{1}{u}\right)=&n_{\textbf{Case2.1}}\left(r,\frac{1}{u}\right)+n_{\textbf{Case2.4}}\left(r,\frac{1}{u}\right)\\\nonumber   
\leq& n_{\textbf{Case2.1}}(r+1,u)+n_{\textbf{Case2.4}}(r+1,u) \\\nonumber  
=&n_{\hat A_{2}}(r+1,u). 
\end{align}
By integrating \eqref{E855}, it follows that 
\begin{align}\label{E856}
N_{A_{1}}\left ( r,\frac{1}{u}\right)\le \left ( \frac{1}{2}+\varepsilon_1  \right ) N_{\hat{A}_1}(r+3,u),\\\nonumber
N_{A_{2}}\left ( r,\frac{1}{u}\right)\le (1+\varepsilon_2)N_{\hat{A}_2}(r+1,u),    
\end{align}
where $\varepsilon_1$ and $\varepsilon_2$ are arbitrary positive constants.

$\textbf{Step 3.}$
Let $u(z_0+2)=0^{1}$. Then the distribution of the zeros and poles of $u(z)$ is as follows:

\textbf{Case 3.1.}
\begin{align*}
&\dots u(z_0-1)=\infty^{1} , u(z_{0})=0^{1} , u(z_0+1)=\infty^{1} , u(z_0+2)=0^{1} ,\\\nonumber
&u(z_0+3)=\infty^{1} ,u(z_0+4)\not=\infty \dots    
\end{align*}

\textbf{Case 3.2.}
\begin{align*}
&\dots u(z_0-1)\not =\infty, u(z_{0})=0^{1}, u(z_0+1)=\infty^{1} , u(z_0+2)=0^{1} ,\\\nonumber
&u(z_0+3)=\infty^{1} ,u(z_0+4)=\infty^{1}  \dots    
\end{align*}

\textbf{Case 3.3.}
\begin{align*}
&\dots u(z_0-1)=\infty^{1}, u(z_{0})=0^{1}, u(z_0+1)=\infty^{1} , u(z_0+2)=0^{1} ,\\\nonumber
&u(z_0+3)=\infty^{1} ,u(z_0+4)=\infty^{1}\dots    
\end{align*}

\textbf{Case 3.4.} 
\begin{align*}
&\dots u(z_0-1)\not =\infty, u(z_{0})=0^{1} , u(z_0+1)=\infty^{1}, u(z_0+2)=0^{1} ,\\\nonumber
&u(z_0+3)=\infty^{1} ,u(z_0+4)\not =\infty \dots    
\end{align*}
We define two sets 
\begin{align}\label{N8}
&B_{1}:=\\\nonumber
&\left \{ z_{0}+2:u(z_0)=0, u(z_0+1)=\infty, u(z_0+2)=0, u(z_0+3)=\infty, u(z_0+4)=\infty \right \} \\\nonumber
\cup&\left \{ z_{0}:u(z_0)=0, u(z_0+1)=\infty, u(z_0+2)=0, u(z_0+3)=\infty, u(z_0+4)=\infty \right \},\\\nonumber  
&B_{2}=: \\\nonumber
&\left \{ z_{0}+2:u(z_0)=0, u(z_0+1)=\infty, u(z_0+2)=0, u(z_0+3)=\infty, u(z_0+4)\neq \infty \right \} \\\nonumber
\cup&\left \{ z_{0}:u(z_0)=0, u(z_0+1)=\infty, u(z_0+2)=0, u(z_0+3)=\infty, u(z_0+4)\neq \infty \right \}.\\\nonumber
\end{align}
We can similarly define two sets $\hat{B}_1$ and $\hat{B}_2,$ which are associated with the two sets $B_{1}$ and $B_{2}$ mentioned above.

\begin{align}\label{N9}
&\hat B_{1}:=\\\nonumber
&\left \{z_0+1:u(z_0)=0,u(z_0+1)=\infty, u(z_0+2)=0, u(z_0+3)=\infty, u(z_0+4)=\infty \right \}\\\nonumber  
 \cup&\left \{z_0+3:u(z_0)=0,u(z_0+1)=\infty, u(z_0+2)=0, u(z_0+3)=\infty, u(z_0+4)=\infty \right \}\\\nonumber
\cup& \left \{z_0+4:u(z_0)=0,u(z_0+1)=\infty, u(z_0+2)=0, u(z_0+3)=\infty, u(z_0+4)=\infty \right \}\\\nonumber
&\hat B_{2}: =\\\nonumber
&\left \{z_0+1:u(z_0)=0,u(z_0+1)=\infty, u(z_0+2)=0, u(z_0+3)=\infty, u(z_0+4)\not =\infty \right \}\\\nonumber  
\cup&\left \{z_0+3:u(z_0)=0,u(z_0+1)=\infty, u(z_0+2)=0, u(z_0+3)=\infty, u(z_0+4)\not =\infty \right \}\\\nonumber
\end{align}
By the definitions \eqref{N8} and \eqref{N9}, 
it is easy to verify that
\begin{align}\label{E866} 
n_{B_1} \left ( r,\frac{1}{u}\right)\le \frac{2}{3}n_{\hat B_1}(r+4,u),\\\nonumber
n_{B_2} \left ( r,\frac{1}{u}\right)\le n_{\hat B_2}(r+3,u).      
\end{align}
By integrating both sides of \eqref{E866}, we obtain that
\begin{align}\label{E867} 
N_{B_1} \left ( r,\frac{1}{u}\right)\le \left ( \frac{2}{3}+\varepsilon _3 \right ) N_{\hat B_1}(r+4,u),\\\nonumber
N_{B_2} \left ( r,\frac{1}{u}\right)\le(1+\varepsilon _4) N_{\hat B_2}(r+3,u).      
\end{align}
where $\varepsilon_3$
and $\varepsilon_4$ are arbitrary positive constants.

$\textbf{Step 4.}$ We assert that 
\begin{align}\label{E857}
N_{A_1}\left ( r,\frac{1}{u}\right)+N_{B_1}\left ( r,\frac{1}{u}\right)+N_{\mathbf{Step 1} }\left ( r,\frac{1}{u}\right )  =S(r,u),  
\end{align}
where the exceptional set associated to the error term $S(r,u)$ is of zero upper density. Assume on the contrary to the assertion that $N_{A_1}\left ( r,\frac{1}{u}\right)+N_{B_1}\left ( r,\frac{1}{u}\right)+N_{\mathbf{Step 1} }\left ( r,\frac{1}{u}\right )  \neq S(r,u)$. Then there exist a constant $k>0$ and a set $F$ of positive upper density, such that \begin{align}\label{E858}
N_{A_1}\left ( r,\frac{1}{u}\right)+N_{B_1}\left ( r,\frac{1}{u}\right)+N_{\mathbf{Step 1} }\left ( r,\frac{1}{u}\right )  \ge kT(r,u)+o(T(r,u))
\end{align}
as $r\longrightarrow \infty$ for all $r\in F.$
Combining \eqref{EN3}, \eqref{E856}, \eqref{E867}, and \eqref{E858}, it follows that \begin{align}\label{E860}
 &\left ( \frac{1}{2}+\varepsilon_1\right)N_{\hat{A}_1}(r+3,u)+\left(\frac{2}{3}+\varepsilon_{3}\right)N_{\hat{B}_{1}}(r+4,u)+\frac{1}{2}N_{\mathbf{Step 1}}(r+3,u)\\\nonumber
\ge&  N_{A_1}\left(r,\frac{1}{u}\right)+N_{B_1}\left(r,\frac{1}{u}\right)+N_{\mathbf{Step 1}}\left (r, \frac{1}{u}\right)+S(r,u)\\\nonumber
\ge& kT(r,u)+S(r,u).
\end{align}
It is easy to verify that 
\begin{align}\label{N4}
&N_{A_{1}}\left ( r,\frac{1}{u}\right)+ N_{A_{2}}\left(r,\frac{1}{u}\right)+ N_{B_{1}}\left(r,\frac{1}{u}\right)+ N_{B_{2}}\left(r,\frac{1}{u}\right)+ N_{\mathbf{Step 1}}\left ( r,\frac{1}{u}\right) \\\nonumber
=&N_{1}\left(r,\frac{1}{u}\right), 
\end{align}
and \begin{align}\label{N5}
&N_{\hat{A}_{1}}\left ( r,u\right)+ N_{\hat{A}_{2}}\left(r,u\right)+ N_{\hat{B}_{1}}\left(r,u\right)+ N_{\hat{B}_{2}}\left(r,u\right)+ N_{\mathbf{Step 1}}\left ( r,u\right) 
\leq N\left(r,u\right). 
\end{align}
By combining \eqref{E337},  \eqref{EN3}, \eqref{E856}, \eqref{E867}, \eqref{N4}, \eqref{N5}, and Lemma \ref{L2}, we obtain 
\begin{align}\label{E861}
&T(r,u)\\\nonumber
=&N_{1}\left(r,\frac{1}{u}\right)+S(r,u)\\\nonumber
=&N_{A_1} \left ( r,\frac{1}{u}\right)+N_{A_2} \left ( r,\frac{1}{u}\right)+N_{B_1} \left ( r,\frac{1}{u}\right)\\\nonumber
&+N_{B_2} \left ( r,\frac{1}{u}\right)+N_{\mathbf{Step 1}}\left ( r,\frac{1}{u}\right)+S(r,u)\\\nonumber
\leq &\left (\frac{1}{2}+\varepsilon_{1}\right)N_{\hat{A}_1} \left ( r+3,u\right)+\left (1+\varepsilon_{2}\right)N_{\hat{A}_2} \left (r+1,u\right)+\left (\frac{2}{3}+\varepsilon_{3}\right)N_{\hat{B}_1} \left ( r+4,u\right)\\\nonumber 
&+\left (1+\varepsilon_{4}\right)N_{\hat{B}_2} \left (r+3,u\right)+\frac{1}{2}N_{\mathbf{Step 1}}\left (r+3,u\right)+S(r,u).\\\nonumber 
\end{align}
Note that the facts are as follows:\begin{align}\label{N6}
\left (\frac{1}{2}+\varepsilon_{1}\right)N_{\hat{A}_{1}}(r+3,u)=(1+\varepsilon_{1})N_{\hat{A}_{1}}(r+3,u)-\frac{1}{2}N_{\hat{A}_{1}}(r+3,u),\\\nonumber
\left (\frac{2}{3}+\varepsilon_{3}\right)N_{\hat{B}_{1}}(r+4,u)=(1+\varepsilon_{3})N_{\hat{B}_{1}}(r+4,u)-\frac{1}{3}N_{\hat{B}_{1}}(r+4,u),\\\nonumber  
\frac{1}{2}N_{\mathbf{Step 1} _{1}}(r+3,u)\leq (1+\varepsilon )N_{\mathbf{Step 1} _{1}}(r+3,u)-\frac{1}{2}N_{\mathbf{Step 1} _{1}}(r+3,u).\\\nonumber
\end{align}
where 
$\varepsilon =\max \left \{\varepsilon_{1}, \varepsilon_{2},\varepsilon_{3}, \varepsilon_{4} \right \}.$
Substituting \eqref{N6} into \eqref{E861}, we have \begin{align}
&T(r,u)\\\nonumber
\leq & (1+\varepsilon)\left [ N_{\hat{A}_{1}}(r+3,u)+N_{\hat{A}_{2}}(r+1,u)+N_{\hat{B}_{1}}(r+4,u) \right ] \\\nonumber
  &+(1+\varepsilon)\left [ N_{\hat{B}_{2}}(r+3,u)+N_{\mathbf{Step 1}}(r+3,u) \right ]-\frac{1}{2}N_{\hat{A}_{1}}(r+3,u)\\\nonumber  
&-\frac{1}{3}N_{\hat{B}_{1}}(r+4,u)-\frac{1}{2}N_{\mathbf{Step 1}}(r+3,u)+S(r,u)\\\nonumber  
\leq & (1+\varepsilon)N(r,u)-\frac{1}{2}N_{\hat{A}_{1}}(r+3,u)-\frac{1}{3}N_{\hat{B}_{1}}(r+4,u)\\\nonumber
&-\frac{1}{2}N_{\mathbf{Step 1}}(r+3,u)+S(r,u)\\\nonumber 
\leq & (1+\varepsilon)T(r,u)-\frac{1}{2}N_{\hat{A}_{1}}(r+3,u)-\frac{1}{3}N_{\hat{B}_{1}}(r+4,u)\\\nonumber
&-\frac{1}{2}N_{\mathbf{Step 1}}(r+3,u)+S(r,u),\\\nonumber 
\end{align}
which implies 
\begin{align*}
\frac{1}{2}N_{\hat{A}_{1}}(r+3,u)+\frac{1}{3}N_{\hat{B}_{1}}(r+4,u)+\frac{1}{2}N_{\mathbf{Step 1}}(r+3,u)\leq \varepsilon T(r,u).
\end{align*}
This contradicts \eqref{E860}. Therefore, the assertion follows. Combining \eqref{E337}, \eqref{E857}, and \eqref{N4}, we have
\begin{align}\label{E865}
N_{A_2}\left ( r,\frac{1}{u}\right)+N_{B_2}\left ( r,\frac{1}{u}\right)=T(r,u)+S(r,u).  
\end{align}

By \eqref{E856}, we find that the distribution of zeros of $u(z)$ is a combination of $A_{2}$ and $B_{2}.$ Additionally, there are at most $S(r,u)$ points that are  considered exceptional. Next, we will proceed to prove that for any $z\in \mathbb{C},$ \begin{align}\label{N7}
\omega(z)=:
u'(z)-u(z)^{2}-C(z)u(z)\neq \infty,   
\end{align}
and there exist at most $S(r,u)$ points such that $
u'(z)-u(z)^{2}-C(z)u(z)= \infty.$

\textbf{Step 5.} All points for which $C(z_{0})=0$ can be included in an error term $S(r,u).$ Therefore, we will only consider cases where $C(z_{0})\neq 0$ in following text. For any $z\in \mathbb{C}$ such that $u(z)\neq \infty,$ the assertion \eqref{N7} is evident. If $u(z)=\infty,$ then by \eqref{E865}, we have $u(z-1)=0.$ 
Let $z=z_{0}$ is an arbitrary generic zero of $u(z).$ Then we only need to prove that 
$$u'(z_{0}+1)-u(z_{0}+1)^{2}-C(z_{0}+1)u(z_{0}+1)\neq \infty.$$
Combining \eqref{E334}, \eqref{E830}, and \eqref{E831}, we have \begin{align}\label{N10}
&u'(z_{0}+1)-u(z_{0}+1)^{2}-C(z_{0}+1)u(z_{0}+1)\\\nonumber
=& \frac{1}{(z-z_{0})^{2}}-\frac{1}{(z-z_{0})^{2}}+ \frac{2\eta_{0}}{z-z_{0}}-C(z_{0}+1)\frac{-1}{z-z_{0}}+O(1)\\\nonumber 
=&\frac{2\eta_{0}+C(z_{0}+1)}{z-z_{0}}+O(1). 
\end{align}
On the other hand, by \eqref{E832} and \eqref{E865}, we conclude that $2\eta_{0}+C(z_{0}+1)=0.$ The assertion \eqref{N7} follows.

\textbf{Step 6.} By \eqref{E835}, we have $m(r,u)=S(r,u).$ By combining this with \eqref{L3}, we can determine that 
\begin{align}\label{N11}
 m(r, u'(z)-u(z)^{2}-C(z)u(z))=S(r,u).   
\end{align}
 Combining \eqref{N7} with \eqref{N11}, it follows that $$T(r,\omega (z))=T(r, u'(z)-u(z)^{2}-C(z)u(z))=S(r,u).$$
From the assumption at the beginning of the proof, we have that   $u(z)=a_0(z)f(z)-a_0(z)a_4(z).$
Therefore, \eqref{N7} becomes
\begin{align*} 
{f}'(z)+L_{1}f(z)+L_{2}f(z)^{2} =L_{3}(z),  
\end{align*}
where 
\begin{align*}
\begin{cases}
 L_{1}(z)=\frac{a'_0(z)+2a_0(z)a_{4}(z)-a_0(z)C(z)}{a_0(z)},\\
 L_{2}(z)=-a_0(z),\\
 L_{3}(z)=\frac{\omega(z)+{a_4}'(z)+a_4(z)^{2}-a_4(z)C(z)}{a_0(z)}, 
\end{cases}
\end{align*}
and $T(r,L_3(z))=S(r,f(z)).$
\end{proof}

\section{Proof of Theorem \ref {T2}}\label{sec7}
The following two lemmas are used to prove Theorem \ref{T2}. If equation \eqref{E8} has solutions that are entire exponential polynomials of the type \eqref{E690}, then we need Lemma \ref{L12} below to analyze the equation. Lemma \ref{L13} below is an auxiliary result needed to prove (ii).

\begin{lemma}\label{L12}\cite[Corollary 4,2]{bib38}
Let $q$ be a positive integer, $a_{0}(z),$ \dots, $a_{n}(z)$  be either exponential polynomials of degree $<q$ or ordinary polynomials in $z,$ and $b_{1},\dots,b_{n}\in \mathbb{C}\setminus \left \{ 0 \right \}$ be distinct constants. Then $$\sum_{j=1}^{n}a_{j}(z)e^{b_{j}z^{q}}=a_{0}(z) $$ holds 
only when $a_{0}(z)\equiv \dots \equiv a_{n}(z)\equiv 0.  $
\end{lemma}

\begin{lemma}\label{L13}\cite [Lemma 1.1.2.]{bib27}
Let $g:(0,+\infty)\longrightarrow \mathbb{R}, h:(0,+\infty)\longrightarrow \mathbb{R}$ be monotone increasing functions such that $g(r)\le h(r)$ outside of an exceptional set $E$ of finite logarithmic measure. Then, for any $\alpha>1,$ there exists $r_{0}>0$ such that $g(r)\le h(r^{\alpha})$ holds for all $r>r_{0} .$
\end{lemma}

\begin{proof}[Proof of Theorem~\ref{T2}]
(i). If at least one of \(a(z)\), \(b(z)\) and \(c(z)\) is a non-constant rational function,  we now assume that $N(r,f(z))\not=S(r,f(z)),$ which implies $f(z)$ has generic poles. Let $E$ be the set of the orders of all generic poles of $f(z).$ The minimum $m$ of the set we call the index of the function $f(z).$ Clearly, the function ${f}'(z)-b(z)f(z)$ has the index $m+1,$ but the index of the right hand side of \eqref{E8} is $m,$ which is a contradiction. We can therefore obtain $$N(r,f(z))=S(r,f(z)).$$ 
Suppose that all coefficient functions of \eqref{E8} are constants. Then all poles of $f(z)$ are generic poles. If $N(r,f(z))\not = 0,$ using  exactly the same reasoning as above, we can find that this is a contradiction.

(ii). 
We  assume that $f(z)$ is a transcendental entire  function with  $\sigma (f)<1.$ 
According to  the Wiman–Valiron theory\cite[p.187-199]{bib25} and \cite[p.28-30]{bib22} and the difference analogue of the Wiman–Valiron theory\cite[Theorem 3]{bib11}
we get 
\begin{align}\label{E621}
\begin{cases}
\frac{{f}'(z)}{f(z)}=\frac{\nu (r,f)}{z}(1+o(1)),\\
\frac{f(z+1)-f(z)}{f(z)}=\frac{\nu (r,f)}{z}(1+o(1)),
\end{cases}
\end{align}
for all $\left | z \right |= r\not \in E\cup [0,1],$ where $E$ is set of  finite logarithmic measure. We can rewrite \eqref{E8} in the form \begin{align}\label{E640}
\frac{{f}'(z)}{f(z)}=a(z)\frac{f(z+1)}{f(z)}+b(z)+\frac{c(z)}{f(z)}.    
\end{align}
By applying \eqref{E621} to \eqref{E640} results in  \begin{align}\label{E641} 
(1-a(z))\left ( \frac{\nu(r,f)}{z}(1+o(1))\right)=a(z)+b(z)+\frac{c(z)}{f(z)}.  
\end{align} 
We can further rewrite \eqref{E641} as \begin{align}\label{E642}
f(z)&=\frac{c(z)}{(1-a(z))\left ( \frac{\nu (r,f)}{z}(1+o(1))\right)-a(z)-b(z)}= \frac{\frac{c(z)}{1-a(z)}}{\frac{\nu (r,f)}{z}(1+o(1))-\frac{a(z)+b(z)}{1-a(z)}  }.  
\end{align}
Note that $a(z),$ $b(z),$ $c(z)$ are rational functions, and $a(z)\not\equiv 1,$ $c(z)\not \equiv 0.$ Then we can obtain the following form   $$\frac{c(z)}{1-a(z)}=Az^{n}(1+o(1)),\frac{a(z)+b(z)}{1-a(z)}=Bz^{m}(1+o(1)),$$ where $A,$ $B$ are complex constants and $m,$ $n$ are integers. On the other hand, due to what we have assumed, $$\sigma(f)=\limsup_{r \to \infty}\frac{\log ^{+}\nu (r,f) }{\log r}<1,$$
 which means $\nu (r,f)=O(r^{L}),L< 1.$ 
 Letting $\left |z\right|=r \longrightarrow \infty (r\notin E \cup \left [ 0,1 \right ])$, equation \eqref{E642} becomes 
\begin{align}\label{E643}
\left|f(z)\right|=&\frac{\left|A\right|\left|z\right|^{n}(1+o(1))}{\left | \frac{\nu (r,f)}{z}(1+o(1))-Bz^{m}(1+o(1))\right|}\\\nonumber  
=&\frac{\left |A\right|\left |z\right|^{n}(1+o(1))}{\left | Bz^{m}(1+o(1))-\frac{\nu (r,f)}{z}(1+o(1))\right|}\\\nonumber   
\le& \frac{\left |A\right|\left|z\right|^{n}(1+o(1))}{\left|B\right|\left |z\right|^{m}(1+o(1))-\frac{\nu(r,f)}{\left|z\right|}(1+o(1)) }\\\nonumber  
=&\left | \frac{A}{B}\right|\left|z\right|^{n-m}(1+o(1)). \end{align}
Therefore, we have \begin{align}\label{E770} 
\underset{r=\left|z\right|\not \in E\cup [0,1] }{M(r,f)}\le 2\left|\frac{A}{B}\right|\left|z\right|^{n-m}.
\end{align}
Next, we state that we can obtain a similar estimate for $M(r,f)$ when the exceptional set $E\cup [0,1]$ is taken into account.
 By Lemma \ref{L13} and \eqref{E770}, it follows that \begin{align}\label{E771} 
M(r,f)\le 2\left |\frac{A}{B}\right|(r^{\alpha})^{n-m}  
\end{align}
holds for all $\left|z\right|=r>r_{0},$ where $\alpha>1.$ 
Combining the generalized Liouville theorem \cite[86-87]{bib43}, \eqref{E771}, and the
Maximum modulus principle
it follows that $f(z)$ is a polynomial when the order of $f(z)$ is strictly less than $1$. This is a contradiction, since we have assumed that $f(z)$ is transcendental. Therefore, the order of $f(z)$ is at least $1.$

(iii) 

(a). Since $c(z)\equiv 0,$
\eqref{E8} can be rewritten in the form 
\begin{align}\label{E644}
\frac{{f}'(z)}{f(z)}=a(z)\frac{f(z+1)}{f(z)}+b(z).
\end{align}
We can see that $f(z)$ must be a entire function by (i). We assume now that $f(z)$ is an transcendental entire  solution with $\sigma(f)<1.$ Then by \eqref{E621}, \eqref{E644} takes the form\begin{align}\label{E645}
\frac{\nu(r,f)}{z}(1+o(1))=a(z)\left [ 1+\frac{\nu (r,f)}{z}(1+o(1))\right]+b(z).
\end{align}
By a simple calculation, \eqref{E645} can be written as \begin{align}\label{E646}
\frac{\nu (r,f)}{z}(1+o(1))=\frac{a(z)+b(z)}{1-a(z)}.  
\end{align}
Letting $r=\left |z\right|\longrightarrow\infty (r\not \in E\cup [0,1]),$ we can easily see that the left-hand side of \eqref{E646} equals 0, which yields a contradiction, since $a(z)+b(z)$ is non-zero complex constant. Hence $\sigma (f)\ge 1.$

(b). Since the non-zero complex constant $k$ is a Borel exceptional value of $f(z),$ we have \begin{align}\label{E647} 
\lambda (f-k)=\limsup_{r\to\infty}\frac{\log N\left(r,\frac{1}{f-k}\right)}{\log r}< \sigma (f)= \limsup_{r\to\infty}\frac{\log T\left(r,f\right)}{\log r}. 
\end{align}
Note that $a(z)$, $b(z)$ and $c(z)$ are constants, and thus all meromorphic solutions of \eqref{E8} must be entire by (i). Therefore $f(z)$ is of regular growth. Assuming that $\lambda (f-k)<\alpha <\beta <\sigma (f),$ we can now rewrite \eqref{E647} as \begin{align}\label{E648} 
\limsup_{r\to\infty}\frac{\log N\left(r,\frac{1}{f-k}\right)}{\log r}<\alpha <\beta <  \limsup_{r\to\infty}\frac{\log T\left(r,f\right)}{\log r}=\lim_{r\to\infty}\frac{\log T\left(r,f\right)}{\log r} . 
\end{align} 
Thus $$\frac{N(r,\frac{1}{f-k})}{T(r,f)}\le \frac{r^{\alpha}}{r^{\beta}}\longrightarrow 0,$$
as $r\longrightarrow \infty.$ So we get $N\left (r,\frac{1}{f-k}\right)=S(r,f).$ By the Weierstrass factorization theorem of entire functions, we may assume that $f(z)=p(z)e^{Q(z)}+k,$ where $p(z)$ is a meromorphic function such that  $T(r,p(z))=S(r,f)$ and $Q(z)$ is a polynomial with degree at most $\sigma(f).$ Then \eqref{E8} can be expressed as 
\begin{align}\label{E649} 
e^{Q(z)}\left [{p}'(z)+p(z){Q}'(z)-ap(z+1)e^{Q(z+1)-Q(z)}-bp(z)\right]=k(a+b)+c. 
\end{align}
Note that $f(z)$ is of non-vanishing finite order, therefore $Q(z)$ cannot be reduced to a constant.
Therefore, the only way that equation \eqref{E649} can remain valid is that both sides are identically zero, and so it follows that $k=\frac{-c}{a+b}.$

(iv). Assume on the contrary to the assertion that \eqref{E8}   has exponential polynomial as its solution, and it satisfies that $Q_{j}(z)$ cannot reduce into a constant for all integer $j$.
Note that $f(z)=P_{1}(z)e^{Q_{1}(z)}+\dots+P_{k}(z)e^{Q_{k}(z)},$ $f(z+1)=P_{1}(z+1)e^{Q_{1}(z+1)}+\dots+P_{k}(z+1)e^{Q_{k}(z+1)},$ $f'(z)=e^{Q_{1}(z)}\left [ P'_{1}(z)+P_{1}(z)Q'_{1}(z)\right]+\dots+e^{Q_{k}(z)}\left [P'_{k}(z)+P_{k}(z)Q'_{k}(z)\right]    .$
Then \eqref{E8} can be written in the form 
\begin{align}\label{E692} 
&e^{Q_{1}(z)}\left [ P'_{1}(z)+P_{1}(z)Q_{1}'(z)-a(z)P_{1}(z+1)\frac{e^{Q_{1}(z+1)}}{e^{Q_{1}(z)}}-b(z)P_{1}(z) \right]\\\nonumber
+&\dots +e^{Q_{k}(z)}\left [ P'_{k}(z)+P_{k}(z)Q_{k}'(z)-a(z)P_{k}(z+1)\frac{e^{Q_{k}(z+1)}}{e^{Q_{k}(z)}}-b(z)P_{k}(z)\right]\\\nonumber 
&= c(z).
\end{align}
Let $P'_{j}(z)+P_{j}(z)Q'_{j}(z)-a(z)P_{j}(z+1)\frac{e^{Q_{j}(z+1)}}{e^{Q_{j}(z)}}-b(z)P_{j}(z)=\Delta_{j}(z),1\le j\le k.$
Then \eqref{E692} can rewritten as 
\begin{align}\label{E707} 
e^{Q_{1}(z)}\Delta_{1}(z)+\dots + e^{Q_{k}(z)}\Delta_{k}(z)=c(z).  
\end{align}
Without loss of generality,
we suppose that \begin{align}\label{E696} 
\deg_{z}Q_{1}(z)\ge\deg_{z}Q_{2}(z)\dots\ge \deg_{z}Q_{k}(z).
\end{align}

\textbf{Case 1.}
We assume that 
\begin{align}\label{E697} 
\deg_{z}Q_{1}(z)>\deg_{z}Q_{2}(z)>\deg_{z}Q_{3}(z)>\dots> \deg_{z}Q_{k}(z).
\end{align}
Then \eqref{E707} takes the following form 
$$e^{Q_{1}(z)}\Delta_{1}(z)=c(z)-e^{Q_{2}(z)}\Delta_{2}(z)-\dots-e^{Q_{k}(z)}\Delta_{k}(z).$$
By using Lemma \ref{L12},  we get \begin{align}\label{E730} 
e^{Q_{2}(z)}\Delta_{2}(z)+\dots+e^{Q_{k}(z)}\Delta_{k}(z)=c(z),
\end{align} 
and \eqref{E730} can further be rewritten as \begin{align}\label{E732} 
e^{Q_{2}(z)}\Delta_{2}(z)=c(z)-e^{Q_{3}(z)}\Delta_{3}(z)-\dots-e^{Q_{k}(z)}\Delta_{k}(z).
\end{align}
By Lemma \ref{L12} again, we have \begin{align}\label{E733} 
e^{Q_{3}(z)}\Delta_{3}(z)+\dots+e^{Q_{k}(z)}\Delta_{k}(z)=c(z).   
\end{align}
By using Lemma \ref{L12} altogether $k-1$ times, we can finally obtain $e^{Q_{k}(z)}\Delta_{k}(z)=c(z).$ Obviously, 
this is a contradiction with $c(z)\not\equiv 0.$ Therefore, our assumption is not valid.

\textbf{Case 2.}
We assume that \begin{align}\label{E712} 
\deg_{z}Q_{1}(z)=\deg_{z}Q_{2}(z)=\deg_{z}Q_{3}(z)=\dots=\deg_{z}Q_{k}(z).
\end{align} 
Then we can rewrite \eqref{E707} as 
\begin{align}\label{E713} 
e^{Q_{1}(z)}(\Delta_1(z)+e^{Q_{2}(z)-Q_{1}(z)}\Delta_{2}(z)+\dots e^{Q_{k}(z)-Q_{1}(z)}\Delta _{k}(z))=c(z). 
\end{align}
Note that the relationship between $\deg_{z}(Q_{j}(z)-Q_{1}(z))$ and $\deg_{z}Q_{1}(z)$ is either $\deg_{z}(Q_{j}(z)-Q_{1}(z))=\deg_{z}Q_{1}(z)$ or $\deg_{z}(Q_{j}(z)-Q_{1}(z))<\deg_{z}Q_{1}(z)$ for all $2\le j \le k.$ In what follows, we discuss these cases separately.

 \textbf{Subcase 2.1.}
Let $\deg_{z}(Q_{j}(z)-Q_{1}(z))<\deg_{z}Q_{1}(z)   $ for all $2\le j\le k.$ Using Lemma \ref{L12} for \eqref{E713}, it follows that $c(z)\equiv 0$. This is in contradiction with assumption $c(z)\not\equiv0.$

\textbf{Subcase 2.2.} 
There exist $Q_{j}(z)$ and $Q_{i}(z)$ such that $\deg_{z}(Q_{j}(z)-Q_{1}(z))=\deg_{z}Q_{1}(z)$ and $\deg_{z}(Q_{i}(z)-Q_{1}(z))<\deg_{z}Q_{1}(z)$ occur at the same time. We suppose now that 
\begin{align}\label{E720}
\begin{cases}
 \deg_{z}(Q_{2}(z)-Q_{1}(z))=\deg_{z}Q_{1}(z),\\
 \deg_{z}(Q_{3}(z)-Q_{1}(z))=\deg_{z}Q_{1}(z),\\
\dots \\    
\deg_{z}(Q_{m}(z)-Q_{1}(z))=\deg_{z}Q_{1}(z),     
\end{cases}
\end{align}
and \begin{align}\label{E721}
\begin{cases}
 \deg_{z}(Q_{m+1}(z)-Q_{1}(z))<\deg_{z}Q_{1}(z),\\
 \deg_{z}(Q_{m+2}(z)-Q_{1}(z))<\deg_{z}Q_{1}(z),\\
\dots \\    
\deg_{z}(Q_{k}(z)-Q_{1}(z))<\deg_{z}Q_{1}(z).     
\end{cases}
\end{align}
Then we can rewrite \eqref{E707} as \begin{align}\label{E722}
&e^{Q_{2}(z)}\Delta_{2}(z)+\dots+e^{Q_{m}(z)}\Delta_{m}(z)\\\nonumber
&\quad+e^{Q_{1}(z)}(\Delta_{1}(z)+e^{Q_{m+1}(z)-Q_1(z)}\Delta_{m+1}(z)+\dots+e^{Q_{k}(z)-Q_1(z)}\Delta_{k}(z))\equiv c(z).   
\end{align}
We can see that either $\deg_{z}(Q_{l}-Q_{2})(3\le l\le m)=\deg_{z}Q_{2}$ or $\deg_{z}(Q_{l}-Q_{2})<\deg_{z}Q_{2}.$
Suppose now that \begin{align}\label{E724}
\begin{cases}
 \deg_{z}(Q_{3}(z)-Q_{2}(z))=\deg_{z}Q_{2}(z),\\
 \deg_{z}(Q_{4}(z)-Q_{2}(z))=\deg_{z}Q_{2}(z),\\
\dots \\    
\deg_{z}(Q_{p}(z)-Q_{2}(z))=\deg_{z}Q_{2}(z),     
\end{cases}
\end{align}
and \begin{align}\label{E725}
\begin{cases}
 \deg_{z}(Q_{p+1}(z)-Q_{2}(z))<\deg_{z}Q_{2}(z),\\
 \deg_{z}(Q_{p+2}(z)-Q_{2}(z))<\deg_{z}Q_{2}(z),\\
\dots \\    
\deg_{z}(Q_{m}(z)-Q_{2}(z))<\deg_{z}Q_{2}(z).     
\end{cases}
\end{align}
Then \eqref{E722} can be further rewritten as 
\begin{align}\label{E723}
&e^{Q_{3}(z)}(\Delta_{3}(z)+e^{Q_{4}(z)-Q_{3}(z)}\Delta_{4}(z) +\dots+e^{Q_{p}(z)-Q_{3}(z)}\Delta_{p}(z))\\\nonumber
+&e^{Q_{2}(z)}(\Delta_{2}(z)+e^{Q_{p+1}(z)-Q_{2}(z)}\Delta _{_{p+1}}(z)+\dots+e^{Q_{m}(z)-Q_{2}(z)}\Delta_{_{m}}(z))\\\nonumber
+&e^{Q_{1}(z)}(\Delta_{1}(z)+e^{Q_{m+1}(z)-Q_1(z)}\Delta_{m+1}(z)+\dots+e^{Q_{k}(z)-Q_1(z)}\Delta_{k}(z))\equiv c(z).   
\end{align}
We now need to consider the relationship between $\deg_{z}(Q_{e}(z)-Q_{3}(z))(4\le e\le p)$ and $\deg_{z}Q_{3}(z).$ Let 
\begin{align}
\begin{cases}
\deg_{z}(Q_{4}(z)-Q_{3}(z))=\deg_{z}Q_{3}(z),\\
\deg_{z}(Q_{5}(z)-Q_{3}(z))=\deg_{z}Q_{3}(z), \\
\dots\\
\deg_{z}(Q_{v}(z)-Q_{3}(z))=\deg_{z}Q_{3}(z), 
\end{cases}
\end{align}
and \begin{align}
\begin{cases}
\deg_{z}(Q_{v+1}(z)-Q_{3}(z))<\deg_{z}Q_{3}(z),\\
\deg_{z}(Q_{v+2}(z)-Q_{3}(z))<\deg_{z}Q_{3}(z), \\
\dots\\
\deg_{z}(Q_{p}(z)-Q_{3}(z))<\deg_{z}Q_{3}(z).
\end{cases}
\end{align}
Then \eqref{E723} can be expressed as 
\begin{align}\label{E740}
&e^{Q_{4}(z)}\Delta_{4}(z)+\dots+e^{Q_{v}(z)\Delta_{v}(z)}\\\nonumber
+&e^{Q_{3}(z)}(\Delta_{3}(z)+e^{Q_{v+1}(z)-Q_{3}(z)}\Delta_{v+1}(z) +\dots+e^{Q_{p}(z)-Q_{3}(z)}\Delta_{p}(z))\\\nonumber
+&e^{Q_{2}(z)}(\Delta_{2}(z)+e^{Q_{p+1}(z)-Q_{2}(z)}\Delta _{_{p+1}}(z)+\dots+e^{Q_{m}(z)-Q_{2}(z)}\Delta_{_{m}}(z))\\\nonumber
+&e^{Q_{1}(z)}(\Delta_{1}(z)+e^{Q_{m+1}(z)-Q_1(z)}\Delta_{m+1}(z)+\dots+e^{Q_{k}(z)-Q_1(z)}\Delta_{k}(z))\equiv c(z).   
\end{align}
By \eqref{E720} and \eqref{E724}
we have $$\deg_{z}(Q_{2}(z)-Q_{1}(z))=\deg_{z}(Q_{3}(z)-Q_{1}(z))=\deg_{z}(Q_{3}(z)-Q_{2}(z))=\deg_{z}Q_{1}(z),    $$ which implies that the leading coefficients of $Q_{1}(z) ,$ $Q_{2}(z) $ and $Q_{3}(z) $ are pairwise different.
Therefore, if we repeat the process, we can finally obtain that \eqref{E723} becomes \begin{align}\label{E727}
\varpi_{1}(z)e^{Q_{1}(z)}+\varpi_{2}(z)e^{Q_{2}(z)}+\dots+\varpi_{m}(z)e^{Q_{m}(z)}=c(z),   
\end{align}
where $\varpi_{j}(z)(1\le j\le m)$ is an exponential polynomial with degree strictly less than $\deg_{z}Q_{1}(z)=\dots=\deg_{z}Q_{k}(z)$ or an ordinary polynomial in $z.$ Note that the leading coefficients of $Q_{1}(z) ,$ \dots $Q_{m}(z)$ are pairwise different. By applying Lemma \ref{L12} to
\eqref{E727} results in  $c(z)\equiv0,$ and this is a contradiction with our assumption. Hence this case is not possible.

\textbf{Case 3.}
Note that case $1$ and case $2$ are not possible. Then this case is a  combination of case $1$ and case $2.$ Without loss of generality, we assume that \begin{align}\label{E735}
&\deg_{z}Q_{1}(z)>\dots>\deg_{z}Q_{l}(z)=\dots =\deg_{z}Q_{p}(z)\\\nonumber
>&\deg_{z}Q_{p+1}(z)>\dots >\deg_{z}Q_{k}(z).
\end{align}
In this case, \eqref{E707} can be written in the following form:
\begin{align}\label{E745}
 &e^{Q_{1}(z)}\Delta_{1}(z)+\dots+e^{Q_{l-1}(z)}\Delta _{l-1}(z)\\\nonumber
+&e^{Q_{l}(z)}(\Delta _{l}(z)+e^{Q_{l+1}(z)-Q_{l}(z)}\Delta _{l+1}(z)+\dots+e^{Q_{p}(z)-Q_{l}(z)}\Delta_{p}(z))\\\nonumber
+&e^{Q_{p+1}(z)}\Delta_{p+1}(z)+\dots+e^{Q_{k}(z)}\Delta_{k}(z)=c(z).\nonumber     
\end{align}
Since $Q_1(z)$ is the polynomial of the highest degree in $z$ for $Q_j(z)(1\le j\le k)$ and $\Delta_{j}(z)(1\le j\le k)  $ are small functions of $e^{Q_{j}(z)} (1\le j\le k),$ it follows that $\Delta_{1}(z)\equiv 0.$
In the same way, we have that  $\Delta_{2}(z)=\dots=\Delta_{l-1}(z)\equiv 0.$
Then \eqref{E745} becomes 
\begin{align}\label{E760}
&e^{Q_{l}(z)}(\Delta_{l}(z)+e^{Q_{l+1}(z)-Q_{l}(z)}\Delta _{l+1}(z)+\dots+e^{Q_{p}(z)-Q_{l}(z)}\Delta_{p}(z))\\\nonumber
=&c(z)-e^{Q_{p+1}(z)}\Delta_{p+1}(z)-\dots-e^{Q_{k}(z)}\Delta_{k}(z).\nonumber     
\end{align}
Note that $\deg_{z}Q_{l}(z)=\deg_{z}Q_{l+1}(z)=\dots=\deg_{z}Q_{p}(z)>\deg_{z}Q_{p+1}(z)>\dots> \deg_{z}Q_{k}(z).$ In order to use Lemma \ref{L12}, we still need to consider the relationship between $\deg_{z}(Q_{l+1}(z)-Q_{l}(z)),\dots , \deg_{z}(Q_{p}(z)-Q_{l}(z))$ and $\deg_{z}Q_{l}(z).$ 
From the proof of case $2$ we can obtain that 

\begin{align}\label{E761}
c(z)-e^{Q_{p+1}(z)}\Delta_{p+1}(z)-\dots-e^{Q_{k}(z)}\Delta_{k}(z)\equiv 0.
\end{align}
For \eqref{E761}, using the fact that $\deg_{z}Q_{p+1}(z)>\dots>\deg_{z}Q_{k}(z)$ and the proof of case 1,
 we have $c(z)\equiv 0$, a contradiction. 
Hence this case is also not possible. The proof of Theorem \ref{T2} is completed.
\end{proof}

\section{Examples}\label{sec8}
Examples \ref{Ex 1} and \ref{Ex 2} illustrate the fact, shown in Theorem~\ref{T3}, that if all admissible meromorphic solutions of \eqref{E10}  are subnormal, then they can be reduced to Riccati differential equation \eqref{E396}.
\begin{example}\label{Ex 1}
$f(z)=z\tan(\pi z)$ solves \eqref{E10} with $a(z)=1,$ $b(z)=\frac{\pi }{z} ,$ $c(z)=-1,$  $d(z)=\pi z,$ i.e., $f(z)=z\tan(\pi z)$ is an entire   solution of

\begin{align*}
{f}'(z)=f(z+1)+\frac{\pi}{z}f(z)^{2}-f(z)+\pi z.      
\end{align*}
It is easy to verify that $f(z)=z\tan(\pi z)$ also solves Riccati differential equation \eqref{E396} with $\gamma (z)=0,$ i.e.,
\begin{align*}
{f}'(z)=\frac{\pi }{z}f(z)^{2}+\frac{1}{z}f(z)+\pi z.    
\end{align*}
\end{example}

\begin{example}\label{Ex 2}
$f(z)=\frac{1}{1+e^{2\pi iz}}+z$ solves the complex Schrödinger equation with delay, equation \eqref{E10}, with $a(z)=-2\pi i(1+2z), b(z)=2\pi i, c(z)=0, d(z)=2\pi i(z^{2}+3z+1)+1$
and \eqref{E396} with $\gamma =2\pi i ,$ that is to say, $f(z)=\frac{1}{1+e^{2\pi iz}}+z$ solves \begin{align*} 
 {f}'(z)=-2\pi i(1+2z)f(z+1)+2\pi if(z)^{2}+2\pi i(z^{2}+3z+1)+1   
\end{align*}
and \begin{align*} 
{f}'(z)=2\pi if(z)^{2}-2\pi i(1+2z)f(z)+2\pi i(z^{2}+z)+1. 
\end{align*}
\end{example}

Example \ref{Ex 4} shows that \eqref{E61} can reduce to \eqref{E610} when all admissible meromorphic solutions of \eqref{E61} are of subnormal growth.
\begin{example}\label{Ex 4} 
$f(z)=\tan(\pi z)+z $ solves \eqref{E61} with \begin{align*}
 \begin{cases}
a_0(z)=\pi, \\
a_1
(z)=-3\pi z-a(z),\\
a_2(z)=a(z)(z-1)+3\pi z^{2}+\pi +1, \\
a_3(z)=-\pi z^{3}+a(z)z-\pi z-z,\\
a_4(z)=z, \\
a(z)=-\pi 
\end{cases} 
\end{align*} and  \eqref{E610}  with \begin{align*}
\begin{cases}
 L_{1}(z)=2\pi z+a(z),\\
L_2(z)=-\pi,\\ 
 L_3(z)=\pi +1
+\pi z^{2}. \end{cases}    
\end{align*}
\end{example}

Example \ref{Ex 3} shows that the exponential polynomial solution of the form \eqref{E690} does exist in \eqref{E8} in such a way that at least one of the polynomials $Q_{j}(z)  (1\le j\le k)$ reduces to a constant.
\begin{example}\label{Ex 3}
$f(z)=e^{z}+e^{2z}+1$ solves the following equation: 
\begin{align*}
{f}'(z)=\frac{1}{e^{2}-e}f(z+1)+\frac{e^{2}-2e}{e^{2}-e}f(z)+\frac{1-e}{e}.    
\end{align*}
\end{example}

\bmhead*{Acknowledgments}

The first author thanks the support of the National Natural Science Foundation of China (\#11871260) and the third author thanks the support of the China Scholarship Council(\#202306820018).

\section*{Declarations}

\bmhead*{Conflict of interest} The authors have no conflicts of interest to declare that are relevant to the content of this article.



\end{document}